\documentclass[11pt,a4paper]{article}
\usepackage[utf8]{inputenc}
\usepackage[english]{babel}
\usepackage{pst-grad} % For gradients
\usepackage{pst-plot} % For axes
\usepackage{pstricks}
\usepackage{amsmath,amssymb,mathrsfs,amsthm}
\usepackage[dvips]{graphicx}
\usepackage{anysize} 
\marginsize{3cm}{2cm}{2cm}{2cm} 
\usepackage{epsfig}
\newcommand{\BB}{\mathbb B}

\newcommand{\CC}{\mathbb C}
\newcommand{\RR}{\mathbb R}

\newcommand{\TT}{\mathbb T}
\newcommand{\pat}{\partial_t}
\newcommand{\pax}{\partial_x}

\newcommand{\paa}{\partial_\alpha}

\usepackage[pdfauthor={Rafael Granero Belinch\'on},pdftitle={},bookmarks,colorlinks]{hyperref}
\hypersetup{colorlinks=true}
\newcounter{comentcount}
\newcounter{teocount}
\newcounter{propcount}
\newtheorem{lem}{Lemma}

\newtheorem{teo}[teocount]{Theorem}  

\newtheorem{NE}{Numerical Evidence}

\newenvironment{coment}
{\stepcounter{comentcount} {\bf \tt Remark} {\bf\tt\arabic{comentcount}} }{ }
\title{Local solvability and turning for the inhomogeneous Muskat problem}
\author{Luigi C. Berselli$^{\mbox{{\footnotesize 1}}}$, Diego C\'ordoba$^{\mbox{{\footnotesize 2}}}$, and Rafael Granero-Belinch\'on$^{\mbox{{\footnotesize 3}}}$}
\begin{document}

\maketitle 

\footnotetext[1]{Email: \texttt{berselli@dma.unipi.it} Dipartimento di matematica applicata "Ulisse Dini", via F. Buonarroti 1/c I-56127 Pisa.}

\footnotetext[2]{Email: \texttt{dcg@icmat.es} Instituto de Ciencias Matem\'aticas CSIC-UAM-UC3M-UCM, Consejo Superior de Investigaciones Cient\'ificas, C/Nicol\'as Cabrera, 13-15, Campus de Cantoblanco, 28049 - Madrid}

\footnotetext[3]{Email: \texttt{rgranero@math.ucdavis.edu} Department of Mathematics, University of California, Davis, CA 95616, USA}

\vspace{0.3cm}
\begin{abstract}
In this work we study the evolution of the free boundary between two different fluids in a porous medium where the permeability is a two dimensional step function. The medium can fill the whole plane $\RR^2$ or a bounded strip $S=\RR\times(-\pi/2,\pi/2)$. The system is in the stable regime if the denser fluid is below the lighter one. First, we show local existence in Sobolev spaces by means of energy method when the system is in the stable regime. Then we prove the existence of curves such that they start in the stable regime and in finite time they reach the unstable one.  This change of regime (turning) was first proven in \cite{ccfgl} for the homogeneus Muskat problem with infinite depth.
\end{abstract}
\vspace{0.3cm}


\textbf{Acknowledgments}: The authors are supported by the Grants MTM2011-26696 and SEV-2011-0087 from Ministerio de Ciencia e Innovaci\'on (MICINN). Diego C\' ordoba was partially supported by StG-203138CDSIF of the ERC. Rafael Granero-Belinch\'on is grateful to Department of Applied Mathematics "Ulisse Dini" of the Pisa University for the hospitality during May-July 2012. We are grateful to Instituto de Ciencias Matem\'aticas (Madrid) and to the Dipartimento di Ingegneria Aerospaziale (Pisa) for computing facilities.

\section{Introduction}
In this work we study the evolution of the interface between two different incompressible fluids with the same viscosity coefficient in a porous medium with two different permeabilities. This problem is of practical importance because it is used as a model for a geothermal reservoir (see \cite{CF} and references therein). The velocity of a fluid flowing in a porous medium satisfies Darcy's law (see \cite{bear,Muskat,bn}) 
\begin{equation}
\frac{\mu}{\kappa(\vec{x})}v=-\nabla p-g\rho(\vec{x}) (0,1),
\label{IIeq1} 
\end{equation}
where $\mu$ is the dynamic viscosity, $\kappa(\vec{x})$ is the permeability of the medium, $g$ is the acceleration due to gravity, $\rho(\vec{x})$ is the density of the fluid, $p(\vec{x})$ is the pressure of the fluid and $v(\vec{x})$ is the incompressible velocity field. In our favourite units, we can assume $g=\mu=1.$

The spatial domains considered in this work are $S=\RR^2,\TT\times\RR$ (infinite depth) and $\RR\times(-\pi/2,\pi/2)$ (finite depth). We have two immiscible and incompressible fluids with the same viscosity and different densities; $\rho^1$ fill in the upper domain $S^1(t)$ and $\rho^2$ fill in the lower domain $S^2(t)$. The curve 
$$
z(\alpha,t)=\{(z_1(\alpha,t),z_2(\alpha,t)): \: \alpha\in\RR\}
$$ 
is the interface between the fluids. In particular we are making the ansatz that $S^1$ and $S^2$ are a partition of $S$ and they are separated by a curve $z$.

The system is in the stable regime if the denser fluid is below the lighter one, \emph{i.e.} $\rho^2>\rho^1$. This is known in the literature as the Rayleigh-Taylor condition. The function that measures this condition is defined as
$$
RT(\alpha,t)=-(\nabla p^2(z(\alpha,t))-\nabla p^1(z(\alpha,t)))\cdot\partial_\alpha^\bot z(\alpha,t)>0.
$$

In the case with $\kappa(\vec{x})\equiv\text{costant}>0$, the motion of a fluid in a two-dimensional porous medium is analogous to the Hele-Shaw cell problem (see \cite{cheng2012global, constantin1999global, escher1997classical, H-S} and the references therein) and if the fluids fill the whole plane (in the case with the same viscosity but different densities) the contour equation satisfies (see \cite{c-g07})
\begin{equation}\label{IIfull}
\pat f=\frac{\rho^2-\rho^1}{2\pi}\text{P.V.}\int_\RR \frac{(\pax f(x)-\pax f(x-\eta))\eta}{\eta^2+(f(x)-f(x-\eta))^2}d\eta.
\end{equation}
They show the existence of classical solution locally in time (see \cite{c-g07} and also \cite{ambrose2004well, e-m10, escher2011generalized, KK}) in the Rayleigh-Taylor stable regime which means that $\rho^2>\rho^1$, and maximum principles for $\|f(t)\|_{L^\infty}$ and $\|\pax f(t)\|_{L^\infty}$ (see \cite{c-g09}). Moreover, in \cite{ccfgl} the authors show that there exists initial data in $H^4$ such that $\|\pax f\|_{L^\infty}$ blows up in finite time. Furthermore, in \cite{castro2012breakdown} the authors prove that there exist analytic initial data in the stable regime for the Muskat problem such that the solution turns to the unstable regime and later no longer belongs to $C^4$. In \cite{ccgs-10} the authors show an energy balance for $L^2$ and that if initially $\|\pax f_0\|_{L^\infty}<1$, then there is global lipschitz solution and if the initial datum has $\|f_0\|_{H^3}<1/5$ then there is global classical solution. In \cite{c-c-g10, SCH} the authors study the case with different viscosities. In \cite{knupfer2010darcy} the authors study the case where the interface reach the boundary in a moving point with a constant (non-zero) angle. 

\begin{figure}[t]
		\begin{center}
		\includegraphics[scale=0.4]{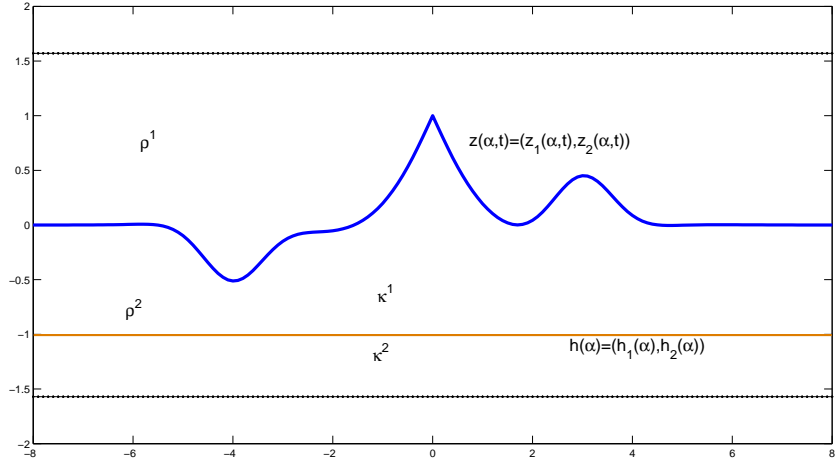} 
		\end{center}
		\caption{Physical situation}
\label{IIscheme}
\end{figure}

The case where the fluid domain is the strip $\RR\times(-l,l)$, with $0<l$, has been studied in \cite{CGO, e-m10, escher2011generalized}. In this regime the equation for the interface is
\begin{multline}
\pat f(x,t) = \frac{\rho^2-\rho^1}{8l}\text{P.V.}\int_\RR\bigg{[}\frac{\left (\partial_xf\left (x\right )-\partial_xf\left (x-\eta\right )\right)\sinh\left(\frac{\pi}{2l}\eta\right)}{\cosh \left(\frac{\pi}{2l}\eta\right)-\cos(\frac{\pi}{2l}(f(x)-f(x-\eta)))}\\
+ \frac{(\partial_xf\left (x\right )+\partial_xf\left (x-\eta\right )\sinh\left(\frac{\pi}{2l}\eta\right)}{\cosh \left(\frac{\pi}{2l}\eta\right)+\cos(\frac{\pi}{2l}(f(x)+f(x-\eta)))}\bigg{]}d\eta.
\label{IIeq0.1}
\end{multline}

For equation \eqref{IIeq0.1} the authors in \cite{CGO} obtain the existence of classical solution locally in time in the stable regime case where the initial interface does not reach the boundaries, and the existence of finite time singularities. These singularities mean that the curve is initially a graph in the stable regime, and in finite time, the curve can not be parametrized as a graph and the interface turns to the unstable regime. Also the authors study the effect of the boundaries on the evolution of the interface, obtaining the maximum principle and a decay estimate for $\|f\|_{L^\infty}$ and the maximum principle for $\|\pax f\|_{L^\infty}$ for initial datum satisfying smallness conditions on $\|\pax f_0\|_{L^\infty}$ and on $\|f_0\|_{L^\infty}$. So, not only the slope must be small, also amplitude of the curve plays a role. Both result differs from the results corresponding to the infinite depth case \eqref{IIfull}. We note that the case with boundaries can also be understood as a problem with different permeabilities where the permeability outside vanishes. In the forthcoming work \cite{GG} the authors compare the different models \eqref{IIfull}, \eqref{IIeq0.1} and \eqref{IIeq13} from the point of view of the existence of turning waves. 

In this work we study the case where permeability $\kappa(\vec{x})$ is a step function, more precisely, we have a curve 
$$
h(\alpha)=\{(h_1(\alpha),h_2(\alpha)): \: \alpha\in\RR\}
$$ 
separating two regions with different values for the permeability (see Figure \ref{IIscheme}). We study the regime with infinite depth, for periodic and for "flat at infinity" initial datum, but also the case where the depth is finite and equal to $\frac{\pi}{2}$. In the region above the curve $h(\alpha)$ the permeability is $\kappa(\vec{x})\equiv\kappa^1$, while in the region below the curve $h(\alpha)$ the permeability is $\kappa(\vec{x})\equiv\kappa^2\neq\kappa^1$. Note that the curve $h(\alpha)$ is known and fixed. Then it follows from Darcy's law that the vorticity is 
$$
\omega(\vec{x})=\varpi_1(\alpha,t)\delta(\vec{x}-z(\alpha,t))+\varpi_2(\alpha,t)\delta(\vec{x}-h(\alpha)),
$$
where $\varpi_1$ corresponds to the difference of the densities, $\varpi_2$ corresponding to the difference of permeabilities and $\delta$ is the usual Dirac's distribution. In fact both amplitudes for the vorticity are quite different, while $\varpi_1$ is a derivative, the amplitude $\varpi_2$ has a nonlocal character (see \eqref{IIeq10A}, \eqref{IIeq14A} and Section \ref{IIsec2}). The equation for the interface, when $h(x)=(x,-h_2)$ and the fluid fill the whole plane, is
\begin{multline}
\label{IIeq9}
\pat f(x)=\frac{\kappa^1(\rho^2-\rho^1)}{2\pi}\text{P.V.}\int_\RR\frac{(\pax f(x)-\pax f(\beta))(x-\beta)}{(x-\beta)^2+(f(x)-f(\beta))^2}d\beta\\
+\frac{1}{2\pi}\text{P.V.}\int_\RR\frac{\varpi_2(\beta)(x-\beta+\pax f(x)(f(x)+h_2))}{(x-\beta)^2+(f(x)+h_2)^2}d\beta,
\end{multline}
with
\begin{eqnarray}
\varpi_2(x)&=&\frac{\kappa^1-\kappa^2}{\kappa^2+\kappa^1}\frac{\kappa^1(\rho^2-\rho^1)}{\pi}\text{P.V.}\int_\RR\frac{\pax f(\beta)(h_2+f(\beta))}{(x-\beta)^2+(-h_2-f(\beta))^2}d\beta\label{IIeq10A}
\end{eqnarray}
If the fluids fill the whole space but the initial curve is periodic the equation reduces to
\begin{multline}\label{IIeq13}
\pat f(x)=\frac{\kappa^1(\rho^2-\rho^1)}{4\pi} \text{P.V.}\int_\TT\frac{\sin(x-\beta)(\pax f(x)-\pax f(\beta))d\beta}{\cosh(f(x)-f(\beta))-\cos(x-\beta)}\\
+\frac{1}{4\pi} \text{P.V.}\int_\TT\frac{(\pax f(x)\sinh(f(x)+h_2)+\sin(x-\beta))\varpi_2(\beta)d\beta}{\cosh(f(x)+h_2)-\cos(x-\beta)},
\end{multline}
where the second vorticity amplitude can be written as
\begin{eqnarray}
\varpi_2(x)&=&\frac{\kappa^1(\rho^2-\rho^1)}{2\pi}\frac{\kappa^1-\kappa^2}{\kappa^1+\kappa^2} \text{P.V.}\int_\TT\frac{\sinh(h_2+f(\beta))\pax f(\beta)d\beta}{\cosh(h_2+f(\beta))-\cos(x-\beta)}.\label{IIeq14A}
\end{eqnarray}

If we consider the regime where the amplitude of the wave and the depth of the medium are of the same order then the equation for the interface, when the depth is chosen to be $\pi/2$, is
\begin{eqnarray}
\pat f(x)&=&\frac{\kappa^1(\rho^2-\rho^1)}{4\pi}\text{P.V.}\int_\RR\frac{(\pax f(x)-\pax f(\beta))\sinh(x-\beta)}{\cosh(x-\beta)-\cos(f(x)-f(\beta))}d\beta\nonumber\\
&&+\frac{\kappa^1(\rho^2-\rho^1)}{4\pi}\text{P.V.}\int_\RR\frac{(\pax f(x)+\pax f(\beta))\sinh(x-\beta)}{\cosh(x-\beta)+\cos(f(x)+f(\beta))}d\beta\nonumber\\
&&+\frac{1}{4\pi}\text{P.V.}\int_\RR\frac{\varpi_2(\beta)(\sinh(x-\beta)+\pax f(x)\sin(f(x)+h_2))}{\cosh(x-\beta)-\cos(f(x)+h_2)}d\beta\nonumber\\
&&+\frac{1}{4\pi}\text{P.V.}\int_\RR\frac{\varpi_2(\beta)(-\sinh(x-\beta)+\pax f(x)\sin(f(x)-h_2))}{\cosh(x-\beta)+\cos(f(x)-h_2)}d\beta\label{IIeqv2},
\end{eqnarray}
where
\begin{eqnarray}
\varpi_2(x)&=&\mathcal{K}\frac{\kappa^1(\rho^2-\rho^1)}{2\pi}\text{P.V.}\int_\RR\pax f(\beta)\frac{\sin(h_2+f(\beta))}{\cosh(x-\beta)-\cos(h_2+f(\beta))}d\beta\nonumber\\
&&-\mathcal{K}\frac{\kappa^1(\rho^2-\rho^1)}{2\pi}\text{P.V.}\int_\RR\pax f(\beta)\frac{\sin(-h_2+f(\beta))}{\cosh(x-\beta)+\cos(-h_2+f(\beta))}d\beta\nonumber\\
&&+\frac{\mathcal{K}^2}{\sqrt{2\pi}}\frac{\kappa^1(\rho^2-\rho^1)}{2\pi}G_{h_2,\mathcal{K}}*\text{P.V.}\int_\RR\frac{\pax f(\beta)\sin(h_2+f(\beta))}{\cosh(x-\beta)-\cos(h_2+f(\beta))}d\beta\nonumber\\
&&-\frac{\mathcal{K}^2}{\sqrt{2\pi}}\frac{\kappa^1(\rho^2-\rho^1)}{2\pi}G_{h_2,\mathcal{K}}*\text{P.V.}\int_\RR\frac{\pax f(\beta)\sin(-h_2+f(\beta))}{\cosh(x-\beta)+\cos(-h_2+f(\beta))}d\beta, \label{IIw2defc}
\end{eqnarray}
with 
$$
G_{h_2,\mathcal{K}}(x)=\mathcal{F}^{-1}\left(\frac{\mathcal{F}\left(\frac{\sin(2h_2)}{\cosh(x)+\cos(2h_2)}\right)(\zeta)}{1+\frac{\mathcal{K}}{\sqrt{2\pi}}\mathcal{F}\left(\frac{\sin(2h_2)}{\cosh(x)+\cos(2h_2)}\right)(\zeta)}\right)
$$
a Schwartz function. 

\begin{coment}
For notational simplicity, we denote $\mathcal{K}=\frac{\kappa^1-\kappa^2}{\kappa^1+\kappa^2}$ and we drop the $t$ dependence.
\end{coment}

The plan of the paper is as follows: in Section \ref{IIsec2} we derive the contour equations \eqref{IIeq9},\eqref{IIeq13} and \eqref{IIeqv2}. In Section \ref{IIsec3} we show the local in time solvability and an energy balance for the $L^2$ norm. In Section \ref{IIsec5} we perform numerics and in Section \ref{IIsec4} we obtain finite time singularities for equations \eqref{IIeq9} \eqref{IIeq13} and \eqref{IIeqv2} when the physical parameters are in some region and numerical evidence showing that, in fact, every value is valid for the physical parameters.

\section{The contour equation}\label{IIsec2}
In this section we derive the contour equations \eqref{IIeq9}, \eqref{IIeq13} and \eqref{IIeqv2}, \emph{i.e.} the equations for the interface. First we obtain the equation in the infinite depth case, both, flat at infinity and periodic. Given $\omega$ a scalar, $\gamma,z,$ curves, and a spatial domain $\Omega=\TT$ or $\Omega=\RR$, we denote the Birkhoff-Rott integral as
\begin{equation}
\label{IIeq2}
BR(\omega,z)\gamma=\text{P.V.}\int_\Omega \omega(\beta) BS(\gamma_1(\alpha),\gamma_2(\alpha),z_1(\beta),z_2(\beta))d\beta,
\end{equation}
where $BS$ denotes the kernel of $\nabla^\perp\Delta^{-1}$ (which depends on the domain). If the domain is $\RR^2$ we have
\begin{equation}\label{IIBSplane}
BS(x,y,\mu,\nu)=\frac{1}{2\pi}\left(-\frac{y-\nu}{(y-\nu)^2+(x-\mu)^2}, \frac{x-\mu}{(y-\nu)^2+(x-\mu)^2}\right),
\end{equation}
for $\TT\times \RR$ we have
\begin{equation}\label{IIBSperiodic}
BS(x,y,\mu,\nu)=\frac{1}{4\pi}\left( \frac{-\sinh(y-\nu)}{\cosh(y-\nu)-\cos(x-\mu)}, \frac{\sin(x-\mu)}{\cosh(y-\nu)-\cos(x-\mu)}\right),
\end{equation}
and for $\RR\times(-\pi/2,\pi/2)$ the kernel is (see \cite{CGO})
\begin{multline}\label{IIBSconf}
BS(x,y,\mu,\nu)=\frac{1}{4\pi}\left(-\frac{\sin(y-\nu)}{\cosh(x-\mu)-\cos(y-\nu)}-\frac{\sin(y+\nu)}{\cosh(x-\mu)+\cos(y+\nu)},\right.\\
\left.\frac{\sinh(x-\mu)}{\cosh(x-\mu)-\cos(y-\nu)}-\frac{\sinh(x-\mu)}{\cosh(x-\mu)+\cos(y+\nu)}\right).
\end{multline}

\subsection{Infinite depth}
\subsubsection{Assuming $S=\RR^2$:} 
Using the kernel \eqref{IIBSplane}, we obtain 
\begin{equation}\label{IIBS}
v(\vec{x})=\frac{1}{2\pi}\text{P.V.}\int_\RR \varpi_1(\beta)\frac{(\vec{x}-z(\beta))^\perp}{|\vec{x}-z(\beta)|^2}d\beta+\frac{1}{2\pi}\text{P.V.}\int_\RR \varpi_2(\beta)\frac{(\vec{x}-h(\beta))^\perp}{|\vec{x}-h(\beta)|^2}d\beta,
\end{equation}
where $(a,b)^\perp=(-b,a).$

We have 
\begin{equation}
\label{IIeq3}
v^\pm(z(\alpha))=\lim_{\epsilon\rightarrow0}v(z(\alpha)\pm\epsilon\paa^\perp z(\alpha))=BR(\varpi_1,z)z+BR(\varpi_2,h)z\mp\frac{1}{2}\frac{\varpi_1(\alpha)}{|\paa z(\alpha)|^2}\paa z(\alpha),
\end{equation}
and
\begin{equation}
\label{IIeq4}
v^\pm(h(\alpha))=\lim_{\epsilon\rightarrow0}v(h(\alpha)\pm\epsilon\paa^\perp h(\alpha))=BR(\varpi_1,z)h+BR(\varpi_2,h)h\mp\frac{1}{2}\frac{\varpi_2(\alpha)}{|\paa h(\alpha)|^2}\paa h(\alpha).
\end{equation}
We observe that $v^+(z(\alpha))$ is the limit inside $S^1$ (the upper subdomain) and $v^-(z(\alpha))$ is the limit inside $S^2$ (the lower subdomain). The curve $z(\alpha)$ doesn't touch the curve $h(\alpha)$, so, the limit for the curve $h$ are in the same domain $S^i$.

Using Darcy's Law and assuming that the initial interface $z(\alpha,0)$ is in the region with permeability $\kappa^1$, we obtain
\begin{eqnarray*}
(v^-(z(\alpha))-v^+(z(\alpha)))\cdot\paa z(\alpha)&=&\kappa^1\left(-\paa(p^-(z(\alpha))-p^+(z(\alpha)))\right)-\kappa^1(\rho^2-\rho^1)\paa z_1(\alpha)\\
&=&0-\kappa^1(\rho^2-\rho^1)\paa z_2(\alpha),
\end{eqnarray*}
where in the last equality we have used the continuity of the pressure along the interface (see \cite{c-c-g10}). Using \eqref{IIeq3} we conclude
\begin{equation}
\label{IIeq5}
\varpi_1(\alpha)=-\kappa^1(\rho^2-\rho^1)\paa z_2(\alpha). 
\end{equation}
We need to determine $\varpi_2$. We consider
\begin{eqnarray*}
\left[\frac{v}{\kappa}\right]&=&\left(\frac{v^-(h(\alpha))}{\kappa^2}-\frac{v^+(h(\alpha))}{\kappa^1}\right)\cdot\paa h(\alpha)\\
&=&-\paa (p^-(h(\alpha))-p^+(h(\alpha)))\\
&=&0,
\end{eqnarray*}
where the first equality is due to Darcy's Law. Using the expression \eqref{IIeq4} we have
\begin{equation}\label{w2eq}
\left[\frac{v}{\kappa}\right]=\left(\frac{1}{\kappa^2}-\frac{1}{\kappa^1}\right)\left(BR(\varpi_1,z)h+BR(\varpi_2,h)h\right)\cdot\paa h(\alpha)+\left(\frac{1}{2\kappa^2}+\frac{1}{2\kappa^1}\right)\varpi_2.
\end{equation}
We take $h(\alpha)=(\alpha,-h_2)$, with $h_2>0$ a fixed constant. Then 
$$
BR(\varpi_2,h)h\cdot\paa h=\left(0,\frac{1}{2}H(\varpi_2)\right)\cdot(1,0)=0,
$$ 
where $H$ denotes the Hilbert transform. Finally, we have
\begin{equation}
\label{IIeq6}
\varpi_2(\alpha)=-2\mathcal{K}BR(\varpi_1,z)h\cdot(1,0)=\mathcal{K}\frac{1}{\pi}\text{P.V.}\int_\RR\varpi_1(\beta)\frac{-h_2-z_2(\beta)}{|h(\alpha)-z(\beta)|^2}d\beta,
\end{equation}
(see Remark 1 for the definition of $\mathcal{K}$).
The identity
$$
\int_{\RR}\partial_\beta\log((A-z_1(\beta))^2+(B-z_2(\beta))^2)=0,
$$
gives us
$$
\frac{1}{2\pi}\text{P.V.}\int_\RR (-\paa z_2(\beta)) \frac{z_2(\alpha)-z_2(\beta)}{|z(\alpha)-z(\beta)|^2}d\beta=\frac{1}{2\pi}\text{P.V.}\int_\RR\paa z_1(\beta) \frac{z_1(\alpha)-z_1(\beta)}{|z(\alpha)-z(\beta)|^2}d\beta,
$$
and
$$
\frac{1}{2\pi}\text{P.V.}\int_\RR \paa z_2(\beta) \frac{h_2+z_2(\beta)}{|h(\alpha)-z(\beta)|^2}d\beta=\frac{1}{2\pi}\text{P.V.}\int_\RR\paa z_1(\beta) \frac{h_1(\alpha)-z_1(\beta)}{|h(\alpha)-z(\beta)|^2}d\beta.
$$
Thus,
\begin{multline}
\label{IIeq7}
\varpi_2(\alpha)=\mathcal{K}\frac{\kappa^1(\rho^2-\rho^1)}{\pi}\text{P.V.}\int_\RR\paa z_2(\beta)\frac{h_2+z_2(\beta)}{|h(\alpha)-z(\beta)|^2}d\beta\\
=\mathcal{K}\frac{\kappa^1(\rho^2-\rho^1)}{\pi}\text{P.V.}\int_\RR\paa z_1(\beta) \frac{h_1(\alpha)-z_1(\beta)}{|h(\alpha)-z(\beta)|^2}d\beta,
\end{multline}
and
$$
BR(\varpi_1,z)z=\frac{-\kappa^1(\rho^2-\rho^1)}{2\pi}\text{P.V.}\int_\RR\frac{z_1(\alpha)-z_1(\beta)}{|z(\alpha)-z(\beta)|^2}\paa z(\beta)d\beta
$$
Due to the conservation of mass the curve $z$ is advected by the flow, but we can add any tangential term in the equation for the evolution of the interface without changing the shape of the resulting curve (see \cite{c-c-g10}), \emph{i.e.} we consider that the equation for the curve is
$$
\pat z(\alpha)=v(\alpha)+c(\alpha,t)\partial_\alpha z(\alpha).
$$
Taking $c(\alpha)=-v_1(\alpha)$, we conclude
\begin{multline}
\label{IIeq8}
\pat z=\frac{\kappa^1(\rho^2-\rho^1)}{2\pi}\text{P.V.}\int_\RR\frac{z_1(\alpha)-z_1(\beta)}{|z(\alpha)-z(\beta)|^2}(\paa z(\alpha)-\paa z(\beta))d\beta\\
+\frac{1}{2\pi}\text{P.V.}\int_\RR\varpi_2(\beta)\frac{(z(\alpha)-h(\beta))^\perp}{|z(\alpha)-h(\beta)|^2}d\beta\\
+\paa z(\alpha)\frac{1}{2\pi}\text{P.V.}\int_\RR\varpi_2(\beta)\frac{z_2(\alpha)+h_2}{|z(\alpha)-h(\beta)|^2}d\beta.
\end{multline}

By choosing this tangential term, if our initial datum can be parametrized as a graph, we have $\pat z_1=0.$ Therefore the parametrization as a graph propagates. 

Finally we conclude \eqref{IIeq9} as the evolution equation for the interface (which initially is a graph above the line $y\equiv-h_2$). We remark that the second vorticity \eqref{IIeq10A} can be written in equivalent ways

\begin{eqnarray}
\varpi_2(x)&=&\mathcal{K}\frac{\kappa^1(\rho^2-\rho^1)}{\pi}\text{P.V.}\int_\RR\pax f(\beta)\frac{h_2+f(\beta)}{(x-\beta)^2+(-h_2-f(\beta))^2}d\beta\label{IIeq10}\\
&=&\mathcal{K}\frac{\kappa^1(\rho^2-\rho^1)}{\pi}\text{P.V.}\int_\RR \frac{x-\beta}{(x-\beta)^2+(-h_2-f(\beta))^2}d\beta\label{IIeq10.b}\\
&=&\mathcal{K}\frac{\kappa^1(\rho^2-\rho^1)}{2\pi}\text{P.V.}\int_\RR\pax\log ((x-\beta)^2+(-h_2-f(\beta))^2)d\beta.\nonumber
\end{eqnarray}

\begin{coment}
Notice that in the case with different viscosities the expression for the amplitude of the vorticity located at the interface $z(\alpha)$ (see equation \eqref{IIeq5}) is no longer valid. Instead, we have
$$
-\kappa^1(\rho^2-\rho^1)\paa z_2(\alpha)=\left(\mu^2-\mu^1\right)\left(BR(\varpi_1,z)z+BR(\varpi_2,h)z\right)\cdot\paa z(\alpha)+\left(\frac{\mu^2+\mu^1}{2}\right)\varpi_1.
$$
To this integral equation, we add the equation \eqref{w2eq} or \eqref{IIeq7}. Thus, one needs to invert an operator. This is a rather delicate issue that is beyond the scope of this paper (see \cite{c-c-g10} for further details in the case $\kappa^1=\kappa^2$).
\end{coment}
\subsubsection{Assuming $S=\TT\times\RR$:} 
We have that \eqref{IIBS} is still valid, but now $\varpi_i$ are periodic functions and $z(\alpha+2k\pi)=z(\alpha)+(2k\pi,0)$. Using complex variables notation we have 
\begin{multline*}
\bar{v}(\vec{x})=\frac{1}{2\pi i}\text{P.V.}\int_\RR\frac{\varpi_1(\beta)}{\vec{x}-z(\beta)}d\beta+\frac{1}{2\pi i}\text{P.V.}\int_\RR\frac{\varpi_2(\beta)}{\vec{x}-h(\beta)}d\beta\\
=\frac{1}{2\pi i}\left(\text{P.V.}\int_{-\pi}^{\pi}+\sum_{k\geq1}\left(\int_{(2k-1)\pi}^{(2k+1)\pi}+\int_{-(2k+1)\pi}^{-(2k-1)\pi}\right)\right)\frac{\varpi_1(\beta)}{\vec{x}-z(\beta)}+\frac{\varpi_2(\beta)}{\vec{x}-h(\beta)}d\beta.
\end{multline*}
Changing variables and using the identity
$$
\frac{1}{z}+\sum_{k\geq1}\frac{2z}{z^2-(2k\pi)^2}=\frac{1}{2\tan(z/2)},\;\;\forall z\in\CC,
$$
we obtain
$$
\bar{v}(\vec{x})=\frac{1}{4\pi i}\left(\text{P.V.}\int_\TT\frac{\varpi_1(\beta)}{\tan((\vec{x}-z(\beta))/2)}d\beta + \text{P.V.}\int_\TT\frac{\varpi_2(\beta)}{\tan((\vec{x}-h(\beta))/2)}d\beta\right).
$$
Equivalently,
\begin{multline*}
v(\vec{x})=\frac{1}{4\pi} \left(\text{P.V.}\int_\TT\frac{-\sinh(y-z_2(\beta))\varpi_1(\beta)d\beta}{\cosh(y-z_2(\beta))-\cos(x-z_1(\beta))}\right.\\
\left.+\text{P.V.}\int_\TT\frac{-\sinh(y-h_2(\beta))\varpi_2(\beta)d\beta}{\cosh(y-h_2(\beta))-\cos(x-h_1(\beta))}\right)\\
+\frac{i}{4\pi}\left( \text{P.V.}\int_\TT\frac{\sin(x-z_1(\beta))\varpi_1(\beta)d\beta}{\cosh(y-z_2(\beta))-\cos(x-z_1(\beta))}\right.\\
\left.+\text{P.V.}\int_\TT\frac{\sin(x-h_1(\beta))\varpi_2(\beta)d\beta}{\cosh(y-h_2(\beta))-\cos(x-h_1(\beta))}\right).
\end{multline*}
Recall that \eqref{IIeq5} and \eqref{IIeq7} are still valid if $h(\alpha)=(\alpha,-h_2)$ for $0<h_2$ a fixed constant. We have
$$
\int_{\TT}\partial_\beta\log(\cosh(B-z_2(\beta))-\cos(A-z_1(\beta)))d\beta=0,
$$
thus, the velocity in the curve when the correct tangential terms are added is
\begin{multline}\label{IIeq11}
\pat z(\alpha)=\frac{1}{4\pi} \left(\kappa^1(\rho^2-\rho^1)\text{P.V.}\int_\TT\frac{\sin(z_1(\alpha)-z_1(\beta))(\paa z(\alpha)-\paa z(\beta))d\beta}{\cosh(z_2(\alpha)-z_2(\beta))-\cos(z_1(\alpha)-z_1(\beta))}\right.\\
\left.+(\paa z_1(\alpha)-1)\text{P.V.}\int_\TT\frac{\sinh(z_2(\alpha)+h_2)\varpi_2(\beta)d\beta}{\cosh(z_2(\alpha)+h_2)-\cos(z_1(\alpha)-h_1(\beta))}\right)\\
+\frac{i}{4\pi} \text{P.V.}\int_\TT\frac{(\paa z_2(\alpha)\sinh(z_2(\alpha)+h_2)+\sin(z_1(\alpha)-h_1(\beta)))\varpi_2(\beta)d\beta}{\cosh(z_2(\alpha)+h_2)-\cos(z_1(\alpha)-h_1(\beta))}.
\end{multline}
We can do the same in order to write $\varpi_2$ as an integral on the torus.
\begin{multline}\label{IIeq12}
\varpi_2(\alpha)=-2\mathcal{K}BR(\varpi_1,z)h\cdot(1,0)\\
=\frac{1}{2\pi}\mathcal{K} \text{P.V.}\int_\TT\frac{\sinh(-h_2-z_2(\beta))\varpi_1(\beta)d\beta}{\cosh(-h_2-z_2(\beta))-\cos(h_1(\alpha)-z_1(\beta))}\\
=\frac{\kappa^1(\rho^2-\rho^1)}{2\pi}\mathcal{K} \text{P.V.}\int_\TT\frac{\sinh(h_2+z_2(\beta))\paa z_2(\beta)d\beta}{\cosh(-h_2-z_2(\beta))-\cos(h_1(\alpha)-z_1(\beta))}.
\end{multline}
If the initial datum can be parametrized as a graph the equation for the interface reduces to \eqref{IIeq13}, where the second vorticity amplitude \eqref{IIeq14A} can be written as
\begin{eqnarray}
\varpi_2(x)&=&\frac{1}{2\pi}\mathcal{K} \text{P.V.}\int_\TT\frac{\sinh(-h_2-f(\beta))\varpi_1(\beta)d\beta}{\cosh(-h_2-f(\beta))-\cos(x-\beta)}\nonumber\\
&=&\frac{\kappa^1(\rho^2-\rho^1)}{2\pi}\mathcal{K} \text{P.V.}\int_\TT\frac{\sinh(h_2+f(\beta))\pax f(\beta)d\beta}{\cosh(h_2+f(\beta))-\cos(x-\beta)}\label{IIeq14}\\
&=&\frac{\kappa^1(\rho^2-\rho^1)}{2\pi}\mathcal{K} \text{P.V.}\int_\TT\frac{\sin(x-\beta)d\beta}{\cosh(h_2+f(\beta))-\cos(x-\beta)}.\label{IIeq14.1}
\end{eqnarray}

\subsection{Finite depth} 
Now we consider the bounded porous medium $\RR\times (-\pi/2,\pi/2)$ (see Figure \ref{IIscheme}). This regime is equivalent to the case with more than two $\kappa^i$ because the boundaries can be understood as regions with $\kappa=0$. As before,
$$
v(x,y)=\text{P.V.}\int_{\RR}\varpi_1(\beta)BS(x,y,z_1(\beta),z_2(\beta))d\beta+\text{P.V.}\int_{\RR}\varpi_2(\beta)BS(x,y,h_1(\beta),h_2(\beta))d\beta.
$$
We assume that $h(\alpha)=(\alpha,-h_2)$ with $0<h_2<\pi/2$. We have that $\varpi_1$ is given by \eqref{IIeq5}. The main difference between the finite depth and the infinite depth is at the level of $\varpi_2$. As in the infinite depth case we have
$$
0=\left(\frac{1}{\kappa^2}-\frac{1}{\kappa^1}\right)\left(BR(\varpi_1,z)h+BR(\varpi_2,h)h\right)\cdot\paa h(\alpha)+\left(\frac{1}{2\kappa^2}+\frac{1}{2\kappa^1}\right)\varpi_2,
$$
where now $BR$ has the usual definition \eqref{IIeq2} in terms of $BS$ in expression \eqref{IIBSconf}. In the unbounded case we have an explicit expression for $\varpi_2$ \eqref{IIeq7} in terms of $z$ and $h$, but now we have a Fredholm integral equation of second kind:
\begin{equation}\label{IIw2def}
\varpi_2(\alpha)+\frac{\mathcal{K}}{2\pi}\;\text{P.V.}\int_\RR\frac{\varpi_2(\beta)\sin(2h_2)}{\cosh(\alpha-\beta)+\cos(2h_2)}d\beta=-2\mathcal{K}BR(\varpi_1,z)h\cdot(1,0).
\end{equation}

After taking the Fourier transform, denoted by $\mathcal{F}(\cdot)(\zeta)$, and using some of its basic properties, we have
\begin{equation*}
\mathcal{F}(\varpi_2)(\zeta)\left(1+\frac{\mathcal{K}}{\sqrt{2\pi}}\mathcal{F}\left(\frac{\sin(2h_2)}{\cosh(x)+\cos(2h_2)}\right)(\zeta)\right)=-2\mathcal{K}\mathcal{F}(BR(\varpi_1,z)h\cdot(1,0))(\zeta).
\end{equation*}
We can solve the equation for $\varpi_2$ for any $|\mathcal{K}|<\delta(h_2)$ with 
\begin{equation}\label{IIdelta}
\delta(h_2)=\min\left\{1,\frac{\sqrt{2\pi}}{\max_{\zeta}\left|\mathcal{F}\left(\frac{\sin(2h_2)}{\cosh(x)+\cos(2h_2)}\right)\right|}\right\}.
\end{equation}
We obtain 
\begin{multline}\label{IIw2defb}
\varpi_2(\alpha)=-2\mathcal{K}BR(\varpi_1,z)h\cdot(1,0)\\
+\frac{2\mathcal{K}^2}{\sqrt{2\pi}}BR(\varpi_1,z)h\cdot(1,0)*\mathcal{F}^{-1}\left(\frac{\mathcal{F}\left(\frac{\sin(2h_2)}{\cosh(x)+\cos(2h_2)}\right)(\zeta)}{1+\frac{\mathcal{K}}{\sqrt{2\pi}}\mathcal{F}\left(\frac{\sin(2h_2)}{\cosh(x)+\cos(2h_2)}\right)(\zeta)}\right).
\end{multline}
Now we observe that if $s(\zeta)$ is a function in the Schwartz class, $\mathcal{S}$, such that $1+s(\zeta)>0$ we have that 
$$
\frac{s(\zeta)}{1+s(\zeta)}\in\mathcal{S},
$$
and we obtain
$$
G_{h_2,\mathcal{K}}(x)=\mathcal{F}^{-1}\left(\frac{\mathcal{F}\left(\frac{\sin(2h_2)}{\cosh(x)+\cos(2h_2)}\right)(\zeta)}{1+\frac{\mathcal{K}}{\sqrt{2\pi}}\mathcal{F}\left(\frac{\sin(2h_2)}{\cosh(x)+\cos(2h_2)}\right)(\zeta)}\right)\in\mathcal{S}.
$$
Recall here that in order to obtain $\varpi_2$ we invert an integral operator. In general this is a delicate issue (compare with \cite{c-c-g10}), but with our choice of $h$ this point can be addressed in a simpler way. Using
$$
\int_\RR\partial_\beta\log\left(\cosh(x-z_1(\beta))\pm\cos(y\pm z_2(\beta))\right)d\beta=0,
$$
and adding the correct tangential term, we obtain
\begin{eqnarray}
\pat z(\alpha)&=&\frac{\kappa^1(\rho^2-\rho^1)}{4\pi}\text{P.V.}\int_\RR\frac{(\paa z(\alpha)-\paa z(\beta))\sinh(z_1(\alpha)-z_1(\beta))}{\cosh(z_1(\alpha)-z_1(\beta))-\cos(z_2(\alpha)-z_2(\beta))}d\beta\nonumber\\
&&+\frac{\kappa^1(\rho^2-\rho^1)}{4\pi}\text{P.V.}\int_\RR\frac{(\paa z_1(\alpha)-\paa z_1(\beta),\paa z_2(\alpha)+\paa z_2(\beta))\sinh(z_1(\alpha)-z_1(\beta))}{\cosh(z_1(\alpha)-z_1(\beta))+\cos(z_2(\alpha)+z_2(\beta))}d\beta\nonumber\\
&&+\frac{1}{4\pi}\text{P.V.}\int_\RR\varpi_2(\beta)BS(z_1(\alpha),z_2(\alpha),\beta,-h_2)d\beta\nonumber\\
&&+\frac{\paa z(\alpha)}{4\pi}\text{P.V.}\int_\RR\varpi_2(\beta)\frac{\sin(z_2(\alpha)+h_2)}{\cosh(z_1(\alpha)-\beta)-\cos(z_2(\alpha)+h_2)}d\beta\nonumber\\ 
&&+\frac{\paa z(\alpha)}{4\pi}\text{P.V.}\int_\RR\varpi_2(\beta)\frac{\sin(z_2(\alpha)-h_2)}{\cosh(z_1(\alpha)-\beta)+\cos(z_2(\alpha)-h_2)}d\beta\label{IIeqv}.
\end{eqnarray}
If the initial curve can be parametrized as a graph the equation reduces to \eqref{IIeqv2} where $\varpi_2$ is defined in \eqref{IIw2defc}.

\begin{coment}
If $h_2=\pi/4$ by an explicit computation we obtain $\delta(\pi/4)=1$, thus, any $\mathcal{K}$ is valid. Moreover, we have tested numerically that the same remains valid for any $0<h_2<\pi/2$, so \eqref{IIw2defc} would be correct for any $\mathcal{K}$. 
\end{coment}

\section{Well-posedness in Sobolev spaces}\label{IIsec3}
\subsection{Energy balance for the $L^2$ norm}\label{IIsec3.0}
Here we obtain an energy balance inequality for the $L^2$ norm of the solution of equation \eqref{IIeqv2}.
We define $\Omega^1=\{(x,y),f(x,t)<y<\pi/2\}$, $\Omega^2=\{(x,y),-h_2<y<f(x,t))\}$ and $\Omega^3=\{(x,y),-\pi/2<y<-h_2\}$.
\begin{lem}For every $0<\kappa^ 1,\kappa^ 2$ the smooth solutions of \eqref{IIeqv2} in the stable regime, \emph{i.e.} $\rho^2>\rho^1$, case verifies
\begin{equation}\label{IIEB2}
\|f(t)\|^2_{L^2(\RR)}+\int_0^t\frac{\|v\|_{L^2(\RR\times(-h_2,\pi/2))}^2}{\kappa^1(\rho^2-\rho^1)}+\frac{\|v\|^2_{L^2(\RR\times(-\pi/2,-h_2))}}{\kappa^2(\rho^2-\rho^1)}ds=\|f_0\|^2_{L^2(\RR)}.
\end{equation}
\end{lem}
\begin{proof}
We define the potentials 
$$
\phi^1(x,y,t)=\kappa^1(p(x,y,t)+\rho^1y),\;\;\text{if }(x,y)\in\Omega^1,
$$  
$$
\phi^2(x,y,t)=\kappa^1(p(x,y,t)+\rho^2y),\;\;\text{if }(x,y)\in\Omega^2,
$$  
$$
\phi^3(x,y,t)=\kappa^2(p(x,y,t)+\rho^2y),\;\;\text{if }(x,y)\in\Omega^3.
$$  
We have $v^i=-\nabla\phi^i$ in each subdomain $S^i$. Since the velocity is incompressible we have
$$
0=\int_{\Omega^i}\Delta\phi^i\phi^i dxdy=-\int_{\Omega^i}|v^i|^2dxdy+\int_{\partial \Omega^i}\phi^i\partial_n\phi^i ds.
$$
Moreover, the normal component of the velocity is continuous through the interface $(x,f(x))$ and the line where permeability changes $(x,-h_2)$. Using the impermeable boundary conditions, we only need to integrate over the curve $(x,f(x,t))$ and $(x,-h_2)$. Indeed, we have
\begin{equation}\label{IIMP1}
0=-\int_{\Omega^1}|v^1|^2dxdy+\kappa^1\int_\RR (p(x,f(x,t),t)+\rho^1f(x,t))(-v(x,f(x,t),t)\cdot(\pax f(x,t),-1))dx,
\end{equation}
\begin{multline}\label{IIMP2}
0=-\int_{\Omega^2}|v^2|^2dxdy+\kappa^1\int_\RR (p(x,f(x,t),t)+\rho^2f(x,t))(-v(x,f(x,t),t)\cdot(-\pax f(x,t),1))dx\\
+\kappa^1\int_\RR (p(x,-h_2,t)-\rho^2h_2)(-v(x,-h_2,t)\cdot(0,-1))dx,
\end{multline}
\begin{equation}\label{IIMP3}
0=-\int_{\Omega^3}|v^3|^2dxdy+\kappa^2\int_\RR (p(x,-h_2,t)-\rho^2h_2)(-v(x,-h_2,t)\cdot(0,1))dx.
\end{equation}
Inserting \eqref{IIMP3} in \eqref{IIMP2} we get
\begin{multline}\label{IIMP2b}
0=-\int_{\Omega^2}|v^2|^2dxdy-\frac{\kappa^1}{\kappa^ 2}\int_{\Omega^3}|v^3|^2dxdy\\
+\kappa^1\int_\RR (p(x,f(x,t),t)+\rho^2f(x,t))(-v(x,f(x,t),t)\cdot(-\pax f(x,t),1))dx,
\end{multline}
Thus, summing \eqref{IIMP2b} and \eqref{IIMP1} together and using the continuity of the pressure and the velocity in the normal direction, we obtain
\begin{equation}\label{IIEB}
\int_{\Omega^ 1\cup \Omega^2}|v|^2dxdy+\frac{\kappa^1}{\kappa^ 2}\int_{\Omega^3}|v|^2dxdy=\kappa^1\int_\RR (\rho^2-\rho^1)f(x,t)(-\pat f(x,t))dx.
\end{equation}
Integrating in time we get the desired result \eqref{IIEB2}.
\end{proof}

\subsection{Well-posedness for the infinite depth case}
Let $\Omega$ be the spatial domain considered, \emph{i.e.} $\Omega=\RR$ or $\Omega=\TT$. In this section we prove the short time existence of classical solution for both spatial domains. We have the following result:
\begin{teo}
Consider $0<h_2$ a fixed constant and the initial datum $f_0(x)=f(x,0)\in H^k(\Omega)$, $k\geq3$, such that $-h_2<\min_x f_0(x)$. Then, if the Rayleigh-Taylor condition is satisfied, \emph{i.e.} $\rho^2-\rho^1>0$, there exists an unique classical solution of \eqref{IIeq9} $f\in C([0,T],H^k(\Omega))$  where $T=T(f_0)$. Moreover, we have $f\in C^1([0,T],C(\Omega))\cap C([0,T],C^2(\Omega)).$
\label{IIteo1}
\end{teo}
\begin{proof}
We prove the result in the case $\Omega=\RR$, being the case $\Omega=\TT$ similar. Let us consider the usual Sobolev space $H^3(\RR)$ endowed with the norm
$$
\|f\|_{H^3}=\|f\|_{L^2}+\|\Lambda^3 f\|_{L^2},
$$
where $\Lambda=\sqrt{-\Delta}$. Define the energy
\begin{equation}\label{IIeq15}
 E[f]:=\|f\|_{H^3}+\|d^h[f]\|_{L^\infty},
\end{equation}
with
\begin{equation}\label{IIeq16}
d^h[f](x,\beta)=\frac{1}{(x-\beta)^2+(f(x)+h_2)^2}.
\end{equation}
To use the classical energy method we need \emph{a priori} estimates. To simplify notation we drop the physical parameters present in the problem by considering $\kappa^1(\rho^2-\rho^1)=2\pi$ and $\mathcal{K}=\frac{1}{2}$. The sign of the difference between the permeabilities will not be important to obtain local existence. We denote $c$ a constant that can changes from one line to another. 

\textbf{Estimates on $\|\varpi_2\|_{H^3}$:} Given $f(x)$ such that $E[f]<\infty$ we consider $\varpi_2$ as defined in \eqref{IIeq10}. Then we have that $\|\varpi_2\|_{H^3}\leq c(E[f]+1)^k$ for some constants $c,k$.  

We proceed now to prove this claim. We start with the $L^2$ norm. Changing variables in \eqref{IIeq10} we have
\begin{multline*}
\|\varpi_2\|^2_{L^2}\leq c\left\|\text{P.V.}\int_{B(0,1)}\frac{\pax f(x-\beta)(h_2+f(x-\beta))}{\beta^2+(h_2+f(x-\beta))^2}d\beta\right\|_{L^2}^2\\
+c\left\|\text{P.V.}\int_{B^c(0,1)}\frac{\pax f(x-\beta)(h_2+f(x-\beta))}{\beta^2+(h_2+f(x-\beta))^2}d\beta\right\|_{L^2}^2\\
=A_1+A_2.\qquad\qquad\qquad\qquad
\end{multline*} 
The inner term, $A_1,$ can be bounded as follows
\begin{multline*}
A_1=\int_\RR\text{P.V.}\int_{B(0,1)}\frac{\pax f(x-\beta)(h_2+f(x-\beta))}{\beta^2+(h_2+f(x-\beta))^2}d\beta dx\\
\times\text{P.V.}\int_{B(0,1)}\frac{\pax f(x-\xi)(h_2+f(x-\xi))}{\xi^2+(h_2+f(x-\xi))^2}d\xi dx\\
\leq c\|d^h[f]\|_{L^\infty}^2(1+\|f\|_{L^\infty})^2\|\pax f\|_{L^2}^2.
\end{multline*}
In the last inequality we have used Cauchy-Schwartz inequality and Tonelli's Theorem. For the outer part we have
\begin{multline*}
A_2=\int_\RR\text{P.V.}\int_{B^c(0,1)}\frac{\pax f(x-\beta)(h_2+f(x-\beta))}{\beta^2+(h_2+f(x-\beta))^2}d\beta dx\\
\times\text{P.V.}\int_{B^c(0,1)}\frac{\pax f(x-\xi)(h_2+f(x-\xi))}{\xi^2+(h_2+f(x-\xi))^2}d\xi dx\\
\leq c(1+\|f\|_{L^\infty})^2\|\pax f\|_{L^2}^2,
\end{multline*}
where we have used that $\int_{1}^\infty\frac{d\beta}{\beta^2}<\infty$ and Cauchy-Schwartz inequality. We change variables in \eqref{IIeq10.b} to obtain
$$
\varpi_2(x)=\text{P.V.}\int_\RR \frac{\beta}{\beta^2+(h_2+f(x-\beta))^2}d\beta.
$$
Now it is clear that $\varpi_2$ is at the level of $f$ in terms of regularity and the inequality follows using the same techniques. Using Sobolev embedding we conclude this step.

\textbf{Estimates on $\|d^h[f]\|_{L^\infty}$:} The first integral in \eqref{IIeq9} can be bounded as follows
$$
I_1\leq\left\|\text{P.V.}\int_\RR\frac{(x-\beta)(\pax f(x)-\pax f(\beta))}{(x-\beta)^2+(f(x)-f(\beta))^2}d\beta\right\|_{L^\infty}\leq c(E[f]+1)^k,
$$
for some positive and finite $k$. The new term is the second integral in \eqref{IIeq9}. 
\begin{multline*}
I_2\leq\left\|\frac{1}{2\pi}\text{P.V.}\int_\RR\frac{\varpi_2(x-\beta)(\beta+\pax f(x)(f(x)+h_2))}{\beta^2+(f(x)+h_2)^2}d\beta\right\|_{L^\infty}\\
\leq \left\|\frac{1}{2\pi}\text{P.V.}\int_{B(0,1)}d\beta\right\|_{L^\infty} +\left\|\frac{1}{2\pi}\text{P.V.}\int_{B^c(0,1)}d\beta\right\|_{L^\infty}= A_1+A_2. 
\end{multline*}
Easily we have
$$
A_1\leq c\|\varpi_2\|_{L^\infty}\|d^h[f]\|_{L^\infty}(1+\|\pax f\|_{L^\infty}(\|f\|_{L^\infty}+1)).
$$
We split $A_2=B_1+B_2$
\begin{multline*}
B_1=\frac{1}{2\pi}\text{P.V.}\int_{B^c(0,1)}\frac{\varpi_2(x-\beta)\beta}{\beta^2+(f(x)+h_2)^2}\pm\frac{\varpi_2(x-\beta)\beta}{\beta^2}d\beta\\
\leq c\|\varpi_2\|_{L^\infty}(\|f\|_{L^\infty}+1)^2+c\|H\varpi_2\|_{L^\infty}+c\|\pax \varpi_2\|_{L^\infty},
\end{multline*}
where $H$ denotes the Hilbert transform. Now we conclude the desired bound using the previous estimate on $\|\varpi_2\|_{H^3}$ and Sobolev embedding. The second term can be bounded as
$$
B_2=\frac{1}{2\pi}\text{P.V.}\int_{B^c(0,1)}\frac{\varpi_2(x-\beta)\pax f(x)(f(x)+h_2)}{\beta^2+(f(x)+h_2)^2}d\beta\leq c\|\varpi_2\|_{L^\infty}(\|f\|_{L^\infty}+1)\|\pax f\|_{L^\infty}.
$$
We obtain the following useful estimate
\begin{equation}\label{IIdtf}
\|\pat f\|_{L^\infty}\leq c(E[f]+1)^k. 
\end{equation}
We have
$$
\frac{d}{dt}d^h[f]=\frac{-\pat f(x)2(f(x)+h_2)}{(\beta^2+(f(x)+h_2)^2)^2}\leq cd^h[f]\|d^h[f]\|_{L^\infty}(\|f\|_{L^\infty}+1)\|\pat f\|_{L^\infty}.
$$
Thus, integrating in time and using \eqref{IIdtf},
$$
\|d^h[f](t+h)\|_{L^\infty}\leq\|d^h[f](t)\|_{L^\infty}e^{c\int_t^{t+h}(E[f]+1)^k},
$$
and we conclude this step
$$
\frac{d}{dt}\|d^h[f]\|_{L^\infty}=\lim_{h\rightarrow0}\frac{\|d^h[f](t+h)\|_{L^\infty}-\|d^h[f](t)\|_{L^\infty}}{h}\leq c(E[f]+1)^k.
$$

\textbf{Estimates on $\|\pax^3 f\|_{L^2}$:} As before, the bound for the term coming from the first integral in \eqref{IIeq9} can be obtained as in \cite{c-g07}, so it only remains the term coming from the second integral. We have
$$
I_2=\frac{1}{2\pi}\int_\RR \pax^3 f(x)\text{P.V.}\int_{\RR}\pax^3\left(\frac{\varpi_2(x-\beta)(\beta+\pax f(x)(f(x)+h_2))}{\beta^2+(f(x)+h_2)^2}\right)d\beta dx.
$$
For the sake of brevity we only bound the terms with higher order, being the remaining terms analogous. We have
$$
I_2=J_3+J_4+J_5+J_6+J_7+\text{l.o.t.},
$$
with
$$
J_3=\frac{1}{2\pi}\int_\RR \pax^3 f(x)\text{P.V.}\int_{\RR}\frac{\pax^3\varpi_2(x-\beta)\beta}{\beta^2+(f(x)+h_2)^2}d\beta dx,
$$
$$
J_4=\frac{1}{2\pi}\int_\RR \pax^3 f(x)\text{P.V.}\int_{\RR}\frac{\pax^3\varpi_2(x-\beta)\pax f(x)(f(x)+h_2)}{\beta^2+(f(x)+h_2)^2}d\beta dx,
$$
$$
J_5=\frac{1}{2\pi}\int_\RR \pax^3 f(x)\text{P.V.}\int_{\RR}\frac{2\varpi_2(x-\beta)(\beta+\pax f(x)(f(x)+h_2))(-f(x)-h_2)\pax^3 f(x)}{(\beta^2+(f(x)+h_2)^2)^2}d\beta dx,
$$
$$
J_6=\frac{1}{2\pi}\int_\RR \pax^3 f(x)\text{P.V.}\int_{\RR}\frac{\varpi_2(x-\beta)(f(x)+h_2)\pax^4 f(x)}{\beta^2+(f(x)+h_2)^2}d\beta dx,
$$
and
$$
J_7=\frac{1}{2\pi}\int_\RR \pax^3 f(x)\text{P.V.}\int_{\RR}\frac{4\varpi_2(x-\beta)\pax f(x)\pax^3 f(x)}{\beta^2+(f(x)+h_2)^2}d\beta dx.
$$
In order to bound $J_3$ we use the symmetries in the formulae ($\pax=-\partial_\beta$) and we integrate by parts:
\begin{multline*}
J_3=\frac{1}{2\pi}\int_\RR \pax^3 f(x)\text{P.V.}\int_{\RR}\pax^2\varpi_2(x-\beta)\partial_\beta\left(\frac{\beta}{\beta^2+(f(x)+h_2)^2}\right)d\beta dx\\
\leq c\|\pax^3 f\|_{L^2}\|\pax^2\varpi_2\|_{L^2}(\|d^h[f]\|_{L^\infty}^2+\|d^h[f]\|_{L^\infty}+1).\hspace{4cm}
\end{multline*}
In $J_4$ we use Cauchy-Schwartz inequality to obtain
$$
J_4\leq c(\|d^h[f]\|_{L^\infty}+1)\|\pax^3 f\|_{L^2}\|\pax^3 \varpi_2\|_{L^2}\|\pax f\|_{L^\infty}(\|f\|_{L^\infty}+h_2)
$$
The bounds for $J_5$ and $J_7$ are similar:
$$
J_5\leq c(\|d^h[f]\|_{L^\infty}^2+1)\|\pax^3 f\|^2_{L^2}\|\varpi_2\|_{L^\infty}(1+\|\pax f\|_{L^\infty}(\|f\|_{L^\infty}+h_2))(\|f\|_{L^\infty}+h_2),
$$ 
$$
J_7\leq c(\|d^h[f]\|_{L^\infty}+1)\|\pax^3 f\|^2_{L^2}\|\varpi_2\|_{L^\infty}\|\pax f\|_{L^\infty}.
$$ 
Finally, we integrate by parts in $J_6$ and we get
\begin{multline*}
J_6\leq c\|\pax^3 f\|^2_{L^2}(\|d^h[f]\|_{L^\infty}+1)\left(\|\pax\varpi_2\|_{L^\infty}(\|f\|_{L^\infty}+1)+\|\varpi_2\|_{L^\infty}\|\pax f\|_{L^\infty}\right)\\
+c\|\pax^3 f\|^2_{L^2}(\|d^h[f]\|_{L^\infty}^2+1)\|\varpi_2\|_{L^\infty}\|\pax f\|_{L^\infty}(\|f\|_{L^\infty}+1)^2.
\end{multline*}
As a conclusion, we obtain
$$
\frac{d}{dt}\|\pax^3 f\|_{L^2}\leq c(E[f]+1)^k.
$$
Putting all the estimates together we get the desired bound for the energy:
\begin{equation}\label{IIenergybound}
\frac{d}{dt}E[f]\leq c(E[f]+1)^k.
\end{equation}

\textbf{Regularization:} This step is classical, so, we only sketch this part (see \cite{bertozzi-Majda} for the details). We regularize the problem and we show that the regularized problems have a solution using Picard's Theorem on a ball in $H^3$. Using the previous energy estimates and the fact that the initial energy is finite, these solutions have the same time of existence ($T$ depending only on the initial datum) and we can show that they are a Cauchy sequence in $C([0,T],L^2)$. From here we obtain $f\in C([0,T],H^s(\Omega))\cap L^\infty([0,T],H^3(\Omega))$  where $T=T(f_0)$ and $0<s<3$, a solution to \eqref{IIeq9} as the limit of these regularized solutions. The continuity of the strongest norm $H^3$ for positive times follows from the parabolic character of the equation. The continuity of $\|f(t)\|_{H^3}$ at $t=0$ follows from the fact that $f(t)\rightharpoonup f_0$ in $H^3$ and from the energy estimates.

\textbf{Uniqueness:} Only remains to show that the solution is unique. Let us suppose that for the same initial datum $f_0$ there are two smooth solutions $f^1$ and $f^2$ with finite energy as defined in \eqref{IIeq15} and consider $f=f^1-f^2$. Following the same ideas as in the energy estimates we obtain
$$
\frac{d}{dt}\|f\|_{L^2}\leq c(f_0,E[f^1],E[f^2])\|f\|_{L^2}.
$$ 
Now we conclude using Gronwall inequality.
\end{proof}

\subsection{Well-posedness for the finite depth case}
In this section we prove the short time existence of classical solution in the case where the depth is finite. We have the following result:
\begin{teo}
Consider $0<h_2<\pi/2$ a constant and $f_0(x)=f(x,0)\in H^k(\RR)$, $k\geq3$, an initial datum such that $\|f_0\|_{L^\infty}<\pi/2$ and $-h_2<\min_x f_0(x)$. Then, if the Rayleigh-Taylor condition is satisfied, \emph{i.e.} $\rho^2-\rho^1>0$, there exists an unique classical solution of \eqref{IIeqv2} $f\in C([0,T],H^k(\RR))$  where $T=T(f_0)$. Moreover, we have $f\in C^1([0,T],C(\RR))\cap C([0,T],C^2(\RR)).$
\label{IIteo3}
\end{teo}
\begin{proof}
Let us consider the usual Sobolev space $H^3(\RR)$, being the other cases analogous, and define the energy
\begin{equation}\label{IIeq23}
 E[f]=\|f\|_{H^3}+\|d^h[f]\|_{L^\infty}+\|d[f]\|_{L^\infty},
\end{equation}
with
\begin{equation}\label{IIeq24}
d^h[f](x,\beta)=\frac{1}{\cosh(x-\beta)-\cos(f(x)+h_2)},
\end{equation}
and
\begin{equation}\label{IIeq25}
d[f](x,\beta)=\frac{1}{\cosh(x-\beta)+\cos(f(x)+f(\beta))}.
\end{equation}
We note that $d^h[f]$ represents the distance between $f$ and $h$ and $d[f]$ the distance between $f$ and the boundaries. To simplify notation we drop the physical parameters present in the problem by considering $\kappa^1(\rho^2-\rho^1)=4\pi$ and $\mathcal{K}=\frac{1}{2}$. Again, the sign of the difference between the permeabilities will not be important to obtain local existence. We write \eqref{IIeqv2} as $\pat f= I_1+I_2+I_3+I_4$, being $I_1,I_2$ the integrals corresponding $\varpi_1$ and $I_3,I_4$ the integrals involving $\varpi_2$. We denote $c$ a constant that can changes from one line to another. 

\textbf{Estimate on $\|\varpi_2\|_{H^3}$:} Given $f(x)$ such that $E[f]<\infty$ and consider $\varpi_2$ as defined in \eqref{IIw2defc}. Then we have that $\|\varpi_2\|_{H^3}\leq c(E[f]+1)^k$. We need to bound $\|J_1\|_{H^3}$ and $\|J_2\|_{H^3}$ with
$$
J_1=\text{P.V.}\int_\RR\pax f(x-\beta)\frac{\sin(h_2+f(x-\beta))}{\cosh(\beta)-\cos(h_2+f(x-\beta))}d\beta
$$
$$
J_2=-\text{P.V.}\int_\RR\pax f(x-\beta)\frac{\sin(-h_2+f(x-\beta))}{\cosh(\beta)+\cos(-h_2+f(x-\beta))}d\beta.
$$
We have
\begin{multline*}
\|J_1\|_{L^2}\leq\left\|\text{P.V.}\int_{B(0,1)}\frac{\pax f(x-\beta)\sin(h_2+f(x-\beta))}{\cosh(\beta)-\cos(h_2+f(x-\beta))}d\beta\right\|_{L^2}\\
+\left\|\text{P.V.}\int_{B^c(0,1)}\frac{\pax f(x-\beta)\sin(h_2+f(x-\beta))}{\cosh(\beta)-\cos(h_2+f(x-\beta))}d\beta\right\|_{L^2}\\
\leq c\|\pax f\|_{L^2}\|d^h[f]\|_{L^\infty}+c\|\pax f\|_{L^2},\hspace{3cm}
\end{multline*}
where we have used Tonelli's Theorem and Cauchy-Schwartz inequality.
Recall that $f-h_2\in \left(-2h_2,\frac{\pi}{2}-h_2\right)$, thus
$$
\frac{1}{\cosh(x-\beta)+\cos(f(x)-h_2)}<\frac{1}{\cosh(x-\beta)-c(h_2)},
$$
and the kernel corresponding to $\varpi_2$ can not be singular and we also obtain
$$
\|J_2\|_{L^2}\leq c\|\pax f\|_{L^2}.
$$
Now, as $G_{h_2,\mathcal{K}}\in\mathcal{S}$, we can use the Young's inequality for the convolution terms obtaining bounds with an universal constant depending on $h_2$ and $\mathcal{K}$. Indeed, we have
$$
\|G_{h_2,\mathcal{K}}*J_i\|_{L^2}\leq c\|J_i\|_{L^2},
$$
and we obtain
$$
\|\varpi_2\|_{L^2}\leq c(E[f]+1)^k.
$$
Now we observe that 
$$
J_1=\text{P.V.}\int_\RR\frac{\sinh(\beta)}{\cosh(\beta)-\cos(h_2+f(x-\beta))}d\beta,\;\;J_2=\text{P.V.}\int_\RR\frac{\sinh(\beta)}{\cosh(\beta)+\cos(-h_2+f(x-\beta))}d\beta,
$$
and we obtain $\|\pax^3 J_i\|_{L^2}\leq c(E[f]+1)^k$. Using Young inequality we conclude
$$
\|\varpi_2\|_{H^3}\leq c(E[f]+1)^k.
$$

\textbf{Estimates on $\|d^h[f]\|_{L^\infty}$ and $\|d[f]\|_{L^\infty}$:} The integrals corresponding to $\varpi_1$ in \eqref{IIeqv2} can be bounded (see \cite{CGO}) as 
$$
|I_1+I_2|\leq c(E[f]+1)^k.
$$
The new terms are the integrals $I_3$ and $I_4$, those involving $\varpi_2$ in \eqref{IIeqv2}. We have, when splitted accordingly to the decay at infinity, 
$$
I_3+I_4=J_3+J_4,
$$
where
\begin{multline*}
|J_3|\leq\left\|\frac{1}{4\pi}\text{P.V.}\int_\RR\frac{\varpi_2(x-\beta)\sinh(\beta)}{\cosh(\beta)-\cos(f(x)+h_2)}-\frac{\varpi_2(x-\beta)\sinh(\beta)}{\cosh(\beta)+\cos(f(x)-h_2)}d\beta\right\|_{L^\infty}\\
\leq c\|\varpi_2\|_{L^\infty}\left(\|d^h[f]\|_{L^\infty}+1\right),
\end{multline*}
and
\begin{multline*}
|J_4|\leq\left\|\frac{1}{4\pi}\text{P.V.}\int_\RR\frac{\varpi_2(x-\beta)\pax f(x)\sin(f(x)+h_2)}{\cosh(\beta)-\cos(f(x)+h_2)}+\frac{\varpi_2(x-\beta)\pax f(x)\sin(f(x)-h_2)}{\cosh(\beta)+\cos(f(x)-h_2)}d\beta\right\|_{L^\infty}\\
\leq c\|\varpi_2\|_{L^\infty}\|\pax f\|_{L^\infty}\left(\|d^h[f]\|_{L^\infty}+1\right).
\end{multline*}
We conclude the following useful estimate
\begin{equation}\label{IIdtfFD}
\|\pat f\|_{L^\infty}\leq c(E[f]+1)^k. 
\end{equation}
We have
$$
\frac{d}{dt}d^h[f]=-\frac{\sin(f(x)+h_2)\pat f(x)}{(\cosh(x-\beta)-\cos(f(x)+h_2))^2}\leq d^h[f]\|d^h[f]\|_{L^\infty}\|\pat f\|_{L^\infty}.
$$
Thus, using \eqref{IIdtfFD} and integrating in time, we obtain the desired bound for $d^h[f]$:
$$
\frac{d}{dt}\|d^h[f]\|_{L^\infty}=\lim_{h\rightarrow0}\frac{\|d^h[f](t+h)\|_{L^\infty}-\|d^h[f](t)\|_{L^\infty}}{h}\leq c(E[f]+1)^k.
$$
To obtain the corresponding bound for $d[f]$ we proceed in the same way and we use \eqref{IIdtfFD} (see \cite{CGO} for the details)

\textbf{Estimates on $\|\pax^3 f\|_{L^2}$:} As before, see \cite{CGO} for the details concerning the terms coming from $\varpi_1$ in \eqref{IIeqv2}. It only remains the terms coming from $\varpi_2$: 
\begin{multline*}
I=\int_\RR \text{P.V.}\int_\RR \pax^3f(x)\pax^3\left(\frac{\varpi_2(\beta)(\sinh(x-\beta)+\pax f(x)\sin(f(x)+h_2))}{\cosh(x-\beta)-\cos(f(x)+h_2)}\right.\\
\left.+\frac{\varpi_2(\beta)(-\sinh(x-\beta)+\pax f(x)\sin(f(x)-h_2))}{\cosh(x-\beta)+\cos(f(x)-h_2)}\right)d\beta dx.
\end{multline*}
We split
$$
I=J_7+J_8+J_9+\text{l.o.t.}.
$$
The lower order terms (l.o.t.) can be obtained in a similar way, so we only study the terms $J_i$. We have
\begin{multline*}
J_7\leq \int_\RR \text{P.V.}\int_\RR \frac{\pax^3 f(x)\pax^3\varpi_2(x-\beta)\sinh(\beta)}{\cosh(\beta)-\cos(f(x)+h_2)}-\frac{\pax^3f(x)\pax^3\varpi_2(x-\beta)\sinh(\beta)}{\cosh(\beta)+\cos(f(x)-h_2)}d\beta dx\\
\leq c\|\pax^3 f\|_{L^2}\|\pax^3 \varpi_2\|_{L^2}(\|d^h[f]+1\|),
\end{multline*}
\begin{multline*}
J_8\leq \int_\RR \text{P.V.}\int_\RR \frac{\pax^3 f(x)\pax^3\varpi_2(x-\beta)\pax f(x)\sin(f(x)+h_2)}{\cosh(\beta)-\cos(f(x)+h_2)}\\
-\frac{\pax^3f(x)\pax^3\varpi_2(x-\beta)\pax f(x)\sin(f(x)+h_2)}{\cosh(\beta)+\cos(f(x)-h_2)}d\beta dx\\
\leq c\|\pax^3 f\|_{L^2}\|\pax^3 \varpi_2\|_{L^2}\|\pax f\|_{L^\infty}(\|d^h[f]+1\|).
\end{multline*}
The term $J_9$ is given by
\begin{multline*}
J_9=\frac{1}{2}\int_\RR \text{P.V.}\int_\RR \pax(\pax^3f(x))^2\left(\frac{\varpi_2(\beta)\sin(f(x)+h_2)}{\cosh(x-\beta)-\cos(f(x)+h_2)}\right.\\
\left.+\frac{\varpi_2(\beta)\sin(f(x)-h_2)}{\cosh(x-\beta)+\cos(f(x)-h_2)}\right)d\beta dx.
\end{multline*}
Integrating by parts
\begin{multline*}
|J_9| \leq c\|\pax^3 f\|_{L^2}(\|d^h[f]\|_{L^\infty}+1)(\|\pax \varpi_2\|_{L^\infty}+\|\varpi_2\|_{L^\infty}\|\pax f\|_{L^\infty})\\
+c\|\pax^3 f\|_{L^2}(\|d^h[f]\|^2_{L^\infty}+1)\|\varpi_2\|_{L^\infty}(1+\|\pax f\|_{L^\infty})
\end{multline*}

\textbf{Regularization and uniqueness:} These steps follow the same lines as in Theorem \ref{IIteo1}. This concludes the result.
\end{proof}

\section{Numerical simulations}\label{IIsec5}
In this section we perform numerical simulations to better understand the role of $\varpi_2$. We consider equation \eqref{IIeq13} where $\kappa^1=1$, $\rho^2-\rho^1=4\pi$ and $h_2=\pi/2$. For each initial datum we approximate the solution of \eqref{IIeq13} corresponding to different $\mathcal{K}$. Indeed, we take different $\kappa^2$ to get $\mathcal{K}=\frac{-999}{1001},\frac{-1}{3},0,\frac{1}{3}$ and $\frac{999}{1001}$.

To perform the simulations we follow the ideas in \cite{c-g-o08}. The interface is approximated using cubic splines with $N$ spatial nodes. The spatial operator is approximated with Lobatto quadrature (using the function \emph{quadl} in Matlab). Then, three different integrals appear for a fixed node $x_i$. The integral between $x_{i-1}$ and $x_i$, the integral between $x_i$ and $x_{i+1}$ and the nonsingular ones. In the two first integrals we use Taylor theorem to remove the zeros present in the integrand. In the nonsingular integrals the integrand is made explicit using the splines. We use a classical explicit Runge-Kutta method of order 4 to integrate in time. In the simulations we take $N=120$ and $dt=10^{-3}$.

The case 1 (see Figure \ref{IIcase1} and \ref{IIcase1b}) approximates the solution corresponding to the initial datum 
$$
f_0(x)=-\left(\frac{\pi}{2}-0.000001\right)e^{-x^{12}}.
$$ 
\begin{figure}[t]
		\begin{center}
		\includegraphics[scale=0.4]{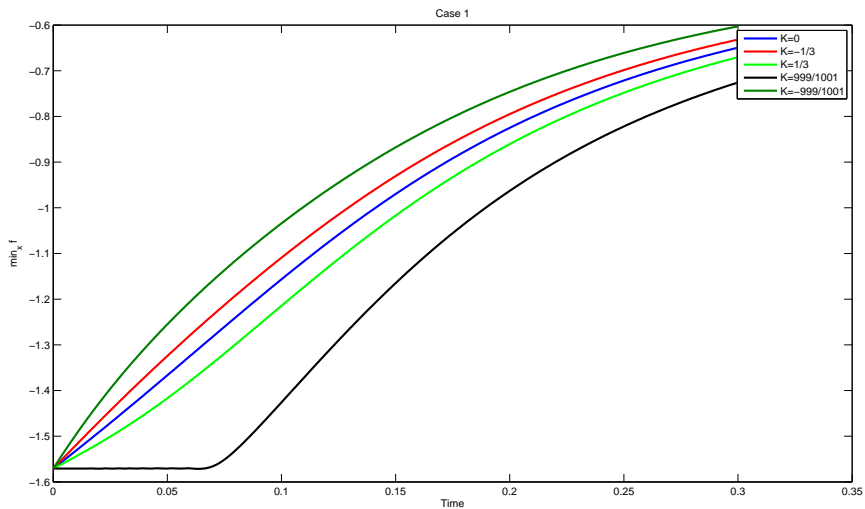} 
		\end{center}
		\caption{Evolution of $-\|f\|_{L^\infty}$ for different $\mathcal{K}$ in case 1.}
\label{IIcase1}
\end{figure}
\begin{figure}[t]
		\begin{center}
		\includegraphics[scale=0.4]{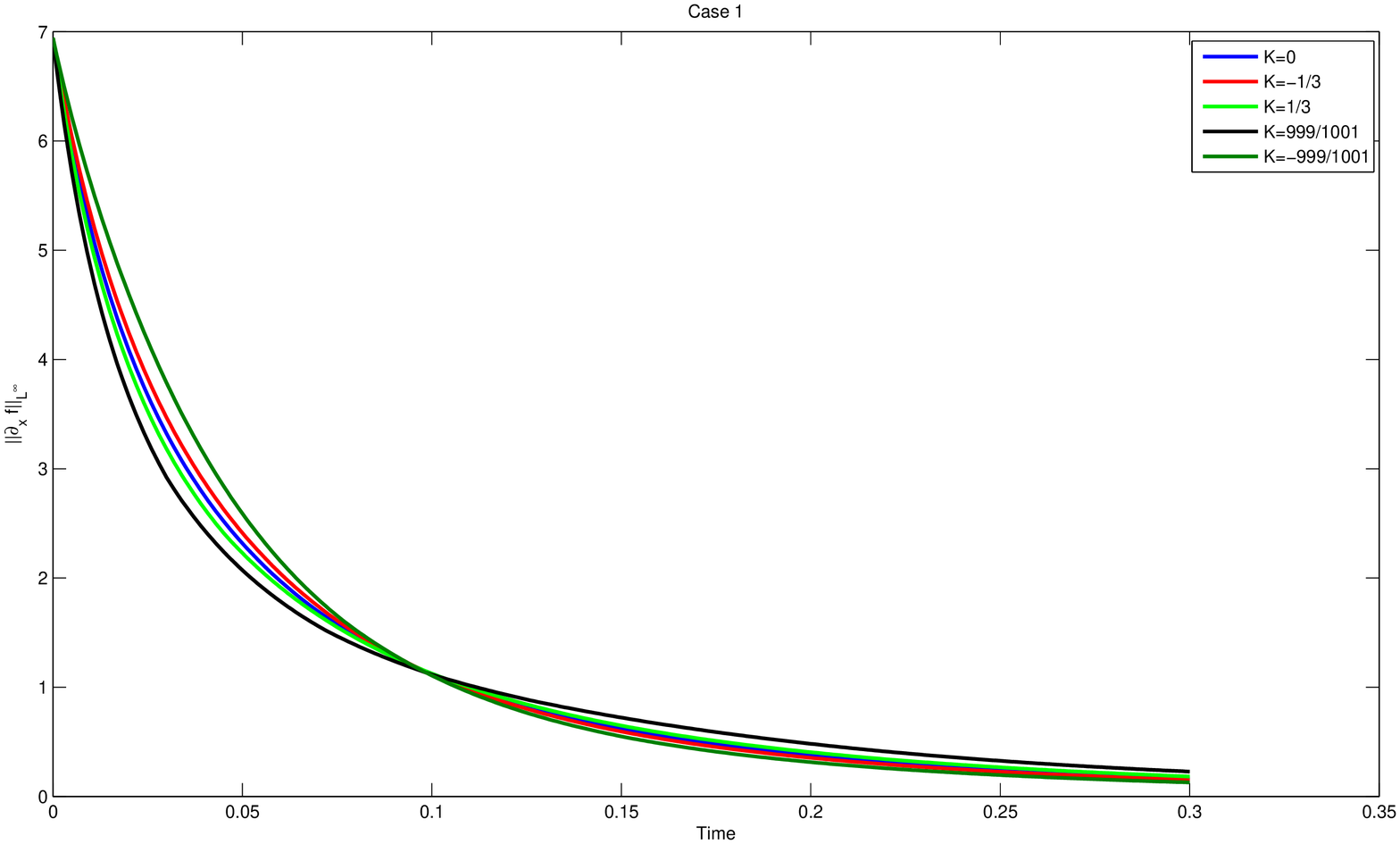} 
		\end{center}
		\caption{Evolution of $\|\pax f\|_{L^\infty}$ for different $\mathcal{K}$ in case 1.}
\label{IIcase1b}
\end{figure}

The case 2 (see Figure \ref{IIcase2} and \ref{IIcase2b}) approximates the solution corresponding to the initial datum 
$$
f_0(x)=-\left(\frac{\pi}{2}-0.000001\right)\cos(x^2).
$$ 
\begin{figure}[t]
		\begin{center}
		\includegraphics[scale=0.4]{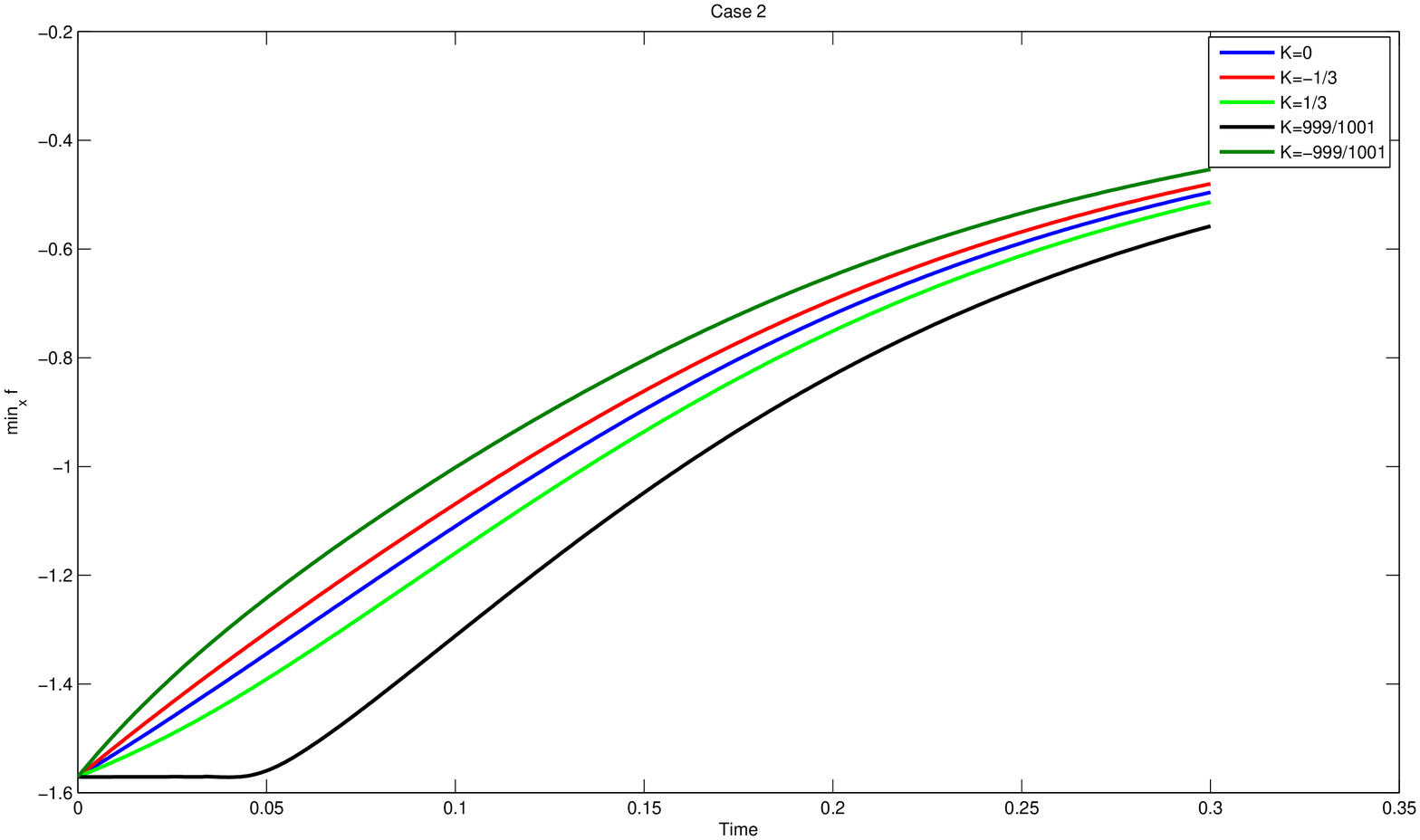} 
		\end{center}
		\caption{Evolution of $-\|f\|_{L^\infty}$ for different $\mathcal{K}$ in case 2.}
\label{IIcase2}
\end{figure}
\begin{figure}[t]
		\begin{center}
		\includegraphics[scale=0.4]{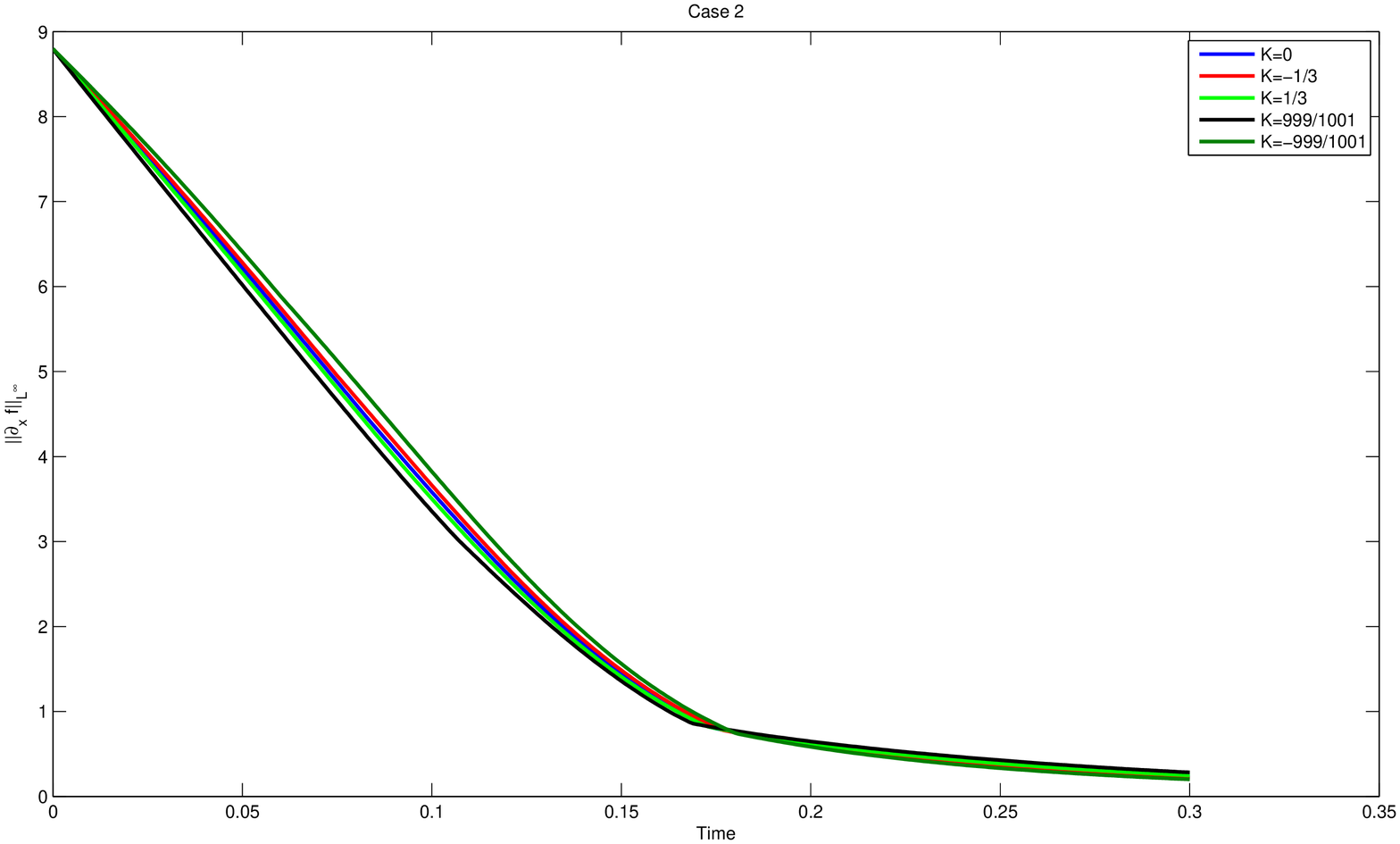} 
		\end{center}
		\caption{Evolution of $\|\pax f\|_{L^\infty}$ for different $\mathcal{K}$ in case 2.}
\label{IIcase2b}
\end{figure}

The case 3 (see Figure \ref{IIcase3} and \ref{IIcase3b}) approximates the solution corresponding to the initial datum 
$$
f_0(x)=-\left(\frac{\pi}{2}-0.000001\right)e^{-(x-2)^{12}}-\left(\frac{\pi}{2}-0.000001\right)e^{-(x+2)^{12}}+e^{-x^2}\cos^2(x).
$$ 
\begin{figure}[t]
		\begin{center}
		\includegraphics[scale=0.4]{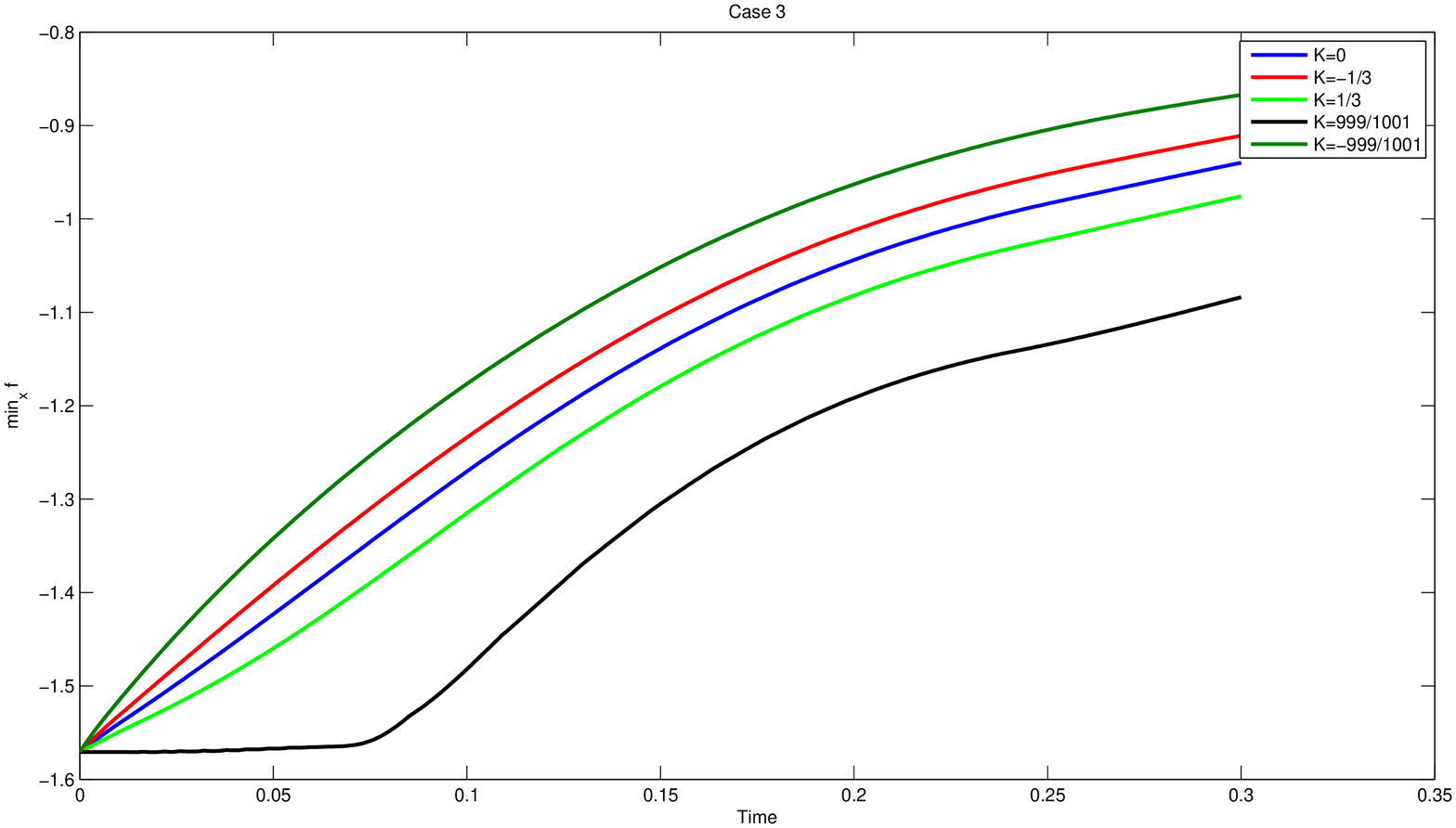} 
		\end{center}
		\caption{Evolution of $-\|f\|_{L^\infty}$ for different $\mathcal{K}$ in case 3.}
\label{IIcase3}
\end{figure}
\begin{figure}[t]
		\begin{center}
		\includegraphics[scale=0.4]{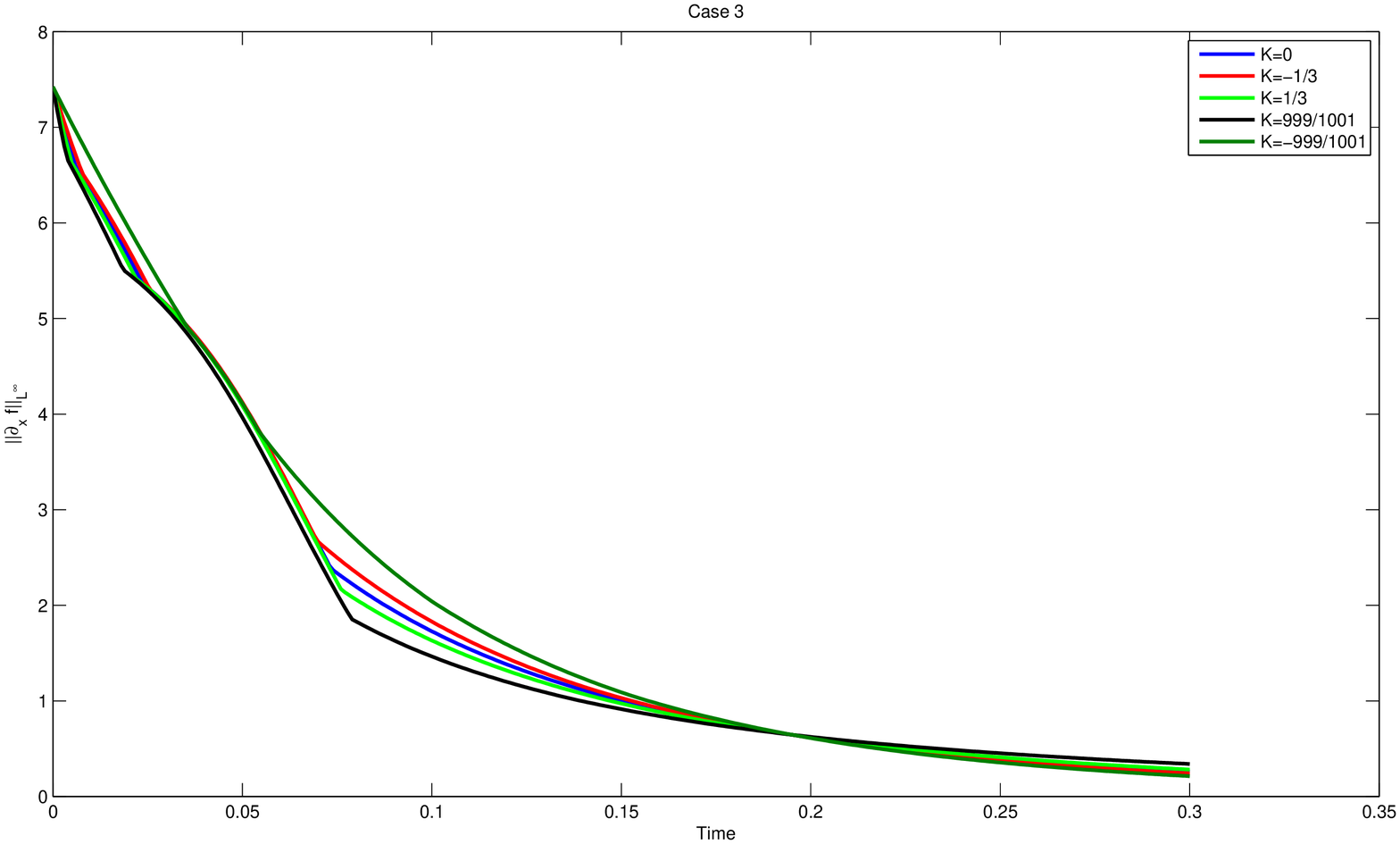} 
		\end{center}
		\caption{Evolution of $\|\pax f\|_{L^\infty}$ for different $\mathcal{K}$ in case 3.}
\label{IIcase3b}
\end{figure}

In these simulations we observe that $\|f\|_{C^1}$ decays but rather differently depending on $\mathcal{K}$. If $\mathcal{K}<0$ the decay of $\|f\|_{L^\infty}$ is faster when compared with the case $\mathcal{K}=0$. In the case where $\mathcal{K}>0$ the term corresponding to $\varpi_2$ slows down the decay of $\|f\|_{L^\infty}$ but we observe still a decay. Particularly, we observe that if $\mathcal{K}\approx 1$ ($\kappa^2\approx 0$) the decay is initially almost zero and then slowly increases. When the evolution of $\|\pax f\|_{L^\infty}$ is considered the situation is reversed. Now the simulations corresponding to $\mathcal{K}>0$ have the faster decay. With these result we can not define a \emph{stable} regime for $\mathcal{K}$ in which the evolution would be \emph{smoother}. Recall that we know that there is not any hypothesis on the sign or size of $\mathcal{K}$ to ensure the existence (see Theorem \ref{IIteo1} and \ref{IIteo3}).

\section{Turning waves}\label{IIsec4}
In this section we prove finite time singularities for equations \eqref{IIeq9}, \eqref{IIeq13} and \eqref{IIeqv2}. These singularities mean that the curve turns over or, equivalently, in finite time they can not be parametrized as graphs. The proof of turning waves follows the steps and ideas in \cite{ccfgl} for the homogeneus infinitely deep case where here we have to deal with the difficulties coming from the boundaries and the delta coming from the jump in the permeabilities.
\subsection{Infinite depth}\label{IIsec4.1}
Let $\Omega$ be the spatial domain considered, \emph{i.e.} $\Omega=\RR$ or $\Omega=\TT$. We have
\begin{teo}
Let us suppose that the Rayleigh-Taylor condition is satisfied, \emph{i.e.} $\rho^2-\rho^1>0$. Then there exists $f_0(x)=f(x,0)\in H^3(\Omega)$, an admissible (see Theorem \ref{IIteo1}) initial datum, such that, for any possible choice of $\kappa^1,\kappa^2>0$ and $h_2>>1$, there exists a solution of \eqref{IIeq9} and \eqref{IIeq13} for which $\lim_{t\rightarrow T^*}\|\pax f(t)\|_{L^\infty}=\infty$ in finite time $0<T^*<\infty$. For short time $t>T^*$ the solution can be continued but it is not a graph.
\label{IIteo2}
\end{teo}
\begin{proof}
To simplify notation we drop the physical parameters present in the problem by considering $\kappa^1(\rho^2-\rho^1)=2\pi$. The proof has three steps. First we consider solutions which are arbitrary curves (not necesary graphs) and we \emph{translate} the singularity formation to the fact $\paa v_1(0)=\pat \paa z_1(0)<0$. The second step is to construct a family of curves such that this expression is negative. Thus, we have that \emph{if there exists, forward and backward in time, a solution in the Rayleigh-Taylor stable case corresponding to initial data which are arbitrary curves} then, we have proved that there is a singularity in finite time. The last step is to prove, using a Cauchy-Kovalevsky theorem, that there exists local in time solutions in this unstable case.

\textbf{Obtaining the correct expression:} Consider the case $\Omega=\RR$. Due to \eqref{IIeq8} we have
$$
\paa \pat z_1(\alpha)=I_1+I_2+I_3,
$$
where
$$
I_1(\alpha)=\paa\text{P.V.}\int_\RR\frac{z_1(\alpha)-z_1(\alpha-\beta)}{|z(\alpha)-z(\alpha-\beta)|^2}(\paa z_1(\alpha)-\paa z_1(\alpha-\beta))d\beta,
$$
$$
I_2(\alpha)=\paa \frac{1}{2\pi}\text{P.V.}\int_\RR\varpi_2(\alpha-\beta)\frac{-z_2(\alpha)-h_2}{|z(\alpha)-h(\alpha-\beta)|^2}d\beta,
$$
$$
I_3(\alpha)=\paa\left(\paa z_1(\alpha)\frac{1}{2\pi}\text{P.V.}\int_\RR\varpi_2(\alpha-\beta)\frac{z_2(\alpha)+h_2}{|z(\alpha)-h(\alpha-\beta)|^2}d\beta\right).
$$
Assume now that the following conditions for $z(\alpha)$ holds:
\begin{itemize}
 \item $z_i(\alpha)$ are odd functions,
 \item $\paa z_1(0)=0,\paa z_1(\alpha)>0$ $\forall \alpha\neq0$, and $\paa z_2(0)>0$,
 \item $z(\alpha)\neq h(\alpha)$ $\forall \alpha$.
\end{itemize}
The previous hypotheses mean that $z$ is a curve satisfying the arc-chord condition and $\paa z(0)$ only has vertical component. Due to these conditions on $z$ we have $\paa z_1(0)=0$ and $\paa^2 z_1$ is odd (and then the second derivative at zero is zero) and we get that $I_3(0)=0$. For $I_1$ we get
\begin{multline*}
I_1(0)=\text{P.V.}\int_\RR\frac{\paa^2z_1(\beta)z_1(\beta)+(\paa z_1(\beta))^2}{(z_1(\beta))^2+(z_2(\beta))^2}d\beta-2\text{P.V.}\int_\RR\frac{(\paa z_1(\beta)z_1(\beta))^2}{((z_1(\beta))^2+(z_2(\beta))^2)^2}d\beta\\
+2\text{P.V.}\int_\RR\frac{\paa z_1(\beta)z_1(\beta)z_2(\beta)\left(\paa z_2(0)-\paa z_2(\beta)\right)}{((z_1(\beta))^2+(z_2(\beta))^2)^2}d\beta.
\end{multline*}
We integrate by parts and we obtain, after some lengthy computations,
\begin{equation}\label{IIsing1R}
I_1(0)=4\paa z_2(0)\text{P.V.}\int_0^\infty\frac{\paa z_1(\beta)z_1(\beta)z_2(\beta)}{((z_1(\beta))^2+(z_2(\beta))^2)^2}d\beta.
\end{equation}
For the term with the second vorticity we have
\begin{multline*}
I_2(0)=\frac{1}{2\pi}\text{P.V.}\int_\RR\frac{\partial_\beta\varpi_2(-\beta)h_2}{\beta^2+h_2^2}d\beta+\frac{1}{2\pi}\text{P.V.}\int_\RR\frac{-\varpi_2(-\beta)\paa z_2(0)}{\beta^2+h_2^2}d\beta\\
-\frac{1}{2\pi}\text{P.V.}\int_\RR\frac{2\varpi_2(-\beta)\beta h_2}{(\beta^2+h_2^2)^2}d\beta-\frac{1}{2\pi}\text{P.V.}\int_\RR\frac{\paa z_2(0)\varpi_2(-\beta)(-h_2^2)}{(\beta^2+h_2^2)^2}d\beta,
\end{multline*}
and, after an integration by parts we obtain
\begin{equation}\label{IIsing2R}
I_2(0)=-\frac{\paa z_2(0)}{2\pi}\text{P.V.}\int_0^\infty\frac{(\varpi_2(\beta)+\varpi_2(-\beta))\beta^2}{(\beta^2+h_2^2)^2}d\beta.
\end{equation}
Putting all together we obtain that in the flat at infinity case the important quantity for the singularity is
\begin{multline}\label{IIsing3R}
\paa v_1(0)=\paa z_2(0)\left(4\text{P.V.}\int_0^\infty\frac{\paa z_1(\beta)z_1(\beta)z_2(\beta)}{((z_1(\beta))^2+(z_2(\beta))^2)^2}d\beta\right.\\
\left.-\frac{1}{2\pi}\text{P.V.}\int_0^\infty\frac{(\varpi_2(\beta)+\varpi_2(-\beta))\beta^2}{(\beta^2+h_2^2)^2}d\beta\right),
\end{multline}
where, due to \eqref{IIeq7}, $\varpi_2$ is defined as
\begin{equation}\label{IIw2f}
\varpi_2(\beta)=2\mathcal{K}\text{P.V.}\int_\RR\frac{(h_2+z_2(\gamma))\paa z_2(\gamma)}{(h_2+z_2(\gamma))^2+(\beta-z_1(\gamma))^2}d\gamma.
\end{equation}

We apply the same procedure to equation \eqref{IIeq11} and we get the importat quantity in the periodic setting (recall the superscript $p$ in the notation denoting that we are in the periodic setting):
\begin{multline}\label{IIsing3T}
\paa v^p_1(0)=\paa z_2(0)\left(\int_0^\pi\frac{\paa z_1(\beta)\sin(z_1(\beta))\sinh(z_2(\beta))}{(\cosh(z_2(\beta))-\cos(z_1(\beta)))^2}d\beta\right.\\
\left.+\frac{1}{4\pi}\int_0^\pi\frac{(\varpi^p_2(\beta)+\varpi^p_2(-\beta))(-1+\cosh(h_2)\cos(\beta))}{(\cosh(h_2)-\cos(\beta))^2}d\beta\right),
\end{multline}
and, due to \eqref{IIeq12}, 
%\begin{equation}\label{IIw2p}
%\varpi^p_2(\beta)=\mathcal{K}\int_0^\pi\frac{\sinh(h_2+z_2(\gamma))\paa z_2(\gamma)}{\cosh(h_2+z_2(\gamma))-\cos(\beta-z_1(\gamma))}+\frac{\sinh(h_2-z_2(\gamma))\paa z_2(\gamma)}{\cosh(h_2-z_2(\gamma))-\cos(\beta+z_1(\gamma))}d\gamma.
%\end{equation}
\begin{equation}\label{IIw2p}
\varpi^p_2(\beta)=\mathcal{K}\int_{\TT}\frac{\sin(\beta-z_1(\gamma))\paa z_1(\gamma)}{\cosh(h_2+z_2(\gamma))-\cos(\beta-z_1(\gamma))}d\gamma.
\end{equation}

\textbf{Taking the appropriate curve:} To clarify the proof, let us consider first the periodic setting. Given $1<h_2$, we consider $a,b,$ constants such that $2<b\leq a$ and let us define
$$
z_1(\alpha)=\alpha-\sin(\alpha),
$$
and
\begin{equation}\label{IIzT}
z_2(\alpha)=\left\{
\begin{array}{lllll}\displaystyle \frac{\sin(a\alpha)}{a} & \hbox{ if }  \displaystyle 0\leq \alpha\leq \frac{\pi}{a},\\
\displaystyle \frac{\sin\left(\pi\frac{\alpha-(\pi/a)}{(\pi/a)-(\pi/b)}\right)}{b} & \text { if } \displaystyle\frac {\pi}{a}<\alpha<\frac{\pi}{b},\\
\displaystyle \left(\frac{-h_2/2}{\frac{\pi}{2}-\frac{\pi}{b}}\right)\left(\alpha -\frac{\pi}{b}\right) & \text{ if }\displaystyle  \frac{\pi}{b}\leq\alpha<\frac{\pi} {2},\\
\displaystyle -\left(\frac{-h_2/2}{\frac{\pi}{2}-\frac{\pi}{b}}\right)\left(\alpha -\pi+\frac{\pi}{b}\right) & \text{ if } 
\displaystyle\frac{\pi}{2}\leq\alpha<\pi (1-\frac{1}{b}),\\
\displaystyle 0 & \text{ if }
\displaystyle \pi(1-\frac{1}{b})\leq\alpha.\\
                   \end{array}\right.
\end{equation}

\begin{figure}[t]
		\begin{center}
		\includegraphics[scale=0.4]{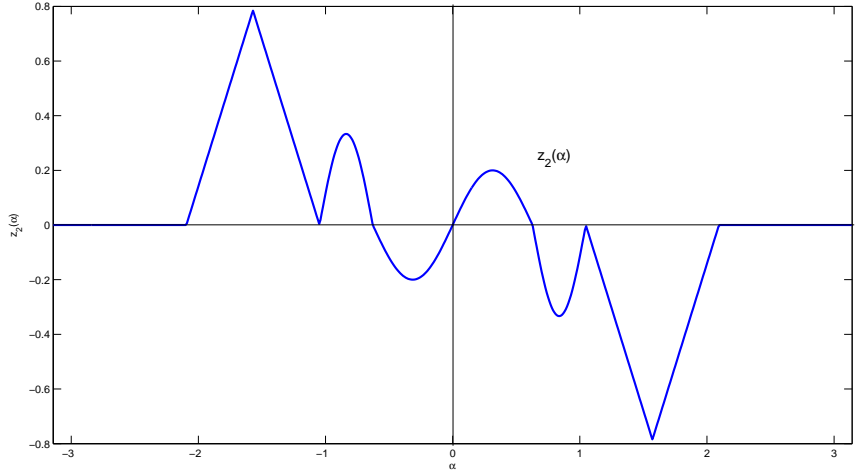} 
		\end{center}
		\caption{$z_2$ in \eqref{IIzT} for $a=5,b=3,h_2=\pi/2$}
\label{IIturning}
\end{figure}

Due to the definition of $z_2$, we have
$$
\frac{h_2}{2}\leq h_2 + z_2(\alpha)\leq \frac{3h_2}{2},
$$
and using \eqref{IIw2p}, we get
$$
|\varpi^p_2(\beta)|\leq \frac{4\pi}{\cosh(h_2/2)-1}.
$$
Inserting this curve in \eqref{IIsing3T} we obtain 
$$
\partial_\alpha v_1^p(0)\leq I_a+I^{h_2}_b+I^{h_2}_2,
$$
with 
$$
I_a=\int_0^{\pi/a}\frac{(1-\cos(\beta))\sin(\beta-\sin(\beta))\sinh\left(\frac{\sin(a\beta)}{a}\right)}{\left(\cosh\left(\frac{\sin(a\beta)}{a}\right)-\cos(\beta-\sin(\beta))\right)^2}d\beta,
$$
\begin{multline*}
I_b^{h_2}=\int_{(\pi/b+\pi)/3}^{\pi/2}\frac{(1-\cos(\beta))\sin(\beta-\sin(\beta))
\sinh\left(\left(\frac{-h_2/2}{\frac{\pi}{2}-\frac{\pi}{b}}\right)\left(\beta -\frac{\pi}{b}\right)\right)}{\left(\cosh\left(\left(\frac{-h_2/2}{\frac{\pi}{2}-\frac{\pi}{b}}\right)\left(\beta -\frac{\pi}{b}\right)\right)-\cos(\beta-\sin(\beta))\right)^2}d\beta\\
+\int_{\pi/2}^{(2\pi-\pi/b)/3)}\frac{(1-\cos(\beta))\sin(\beta-\sin(\beta))
\sinh\left(-\left(\frac{-h_2/2}{\frac{\pi}{2}-\frac{\pi}{b}}\right)\left(\alpha -\pi+\frac{\pi}{b}\right)\right)}{\left(\cosh\left(-\left(\frac{-h_2/2}{\frac{\pi}{2}-\frac{\pi}{b}}\left(\alpha -\pi+\frac{\pi}{b}\right)\right)\right)-\cos(\beta-\sin(\beta))\right)^2}d\beta,
\end{multline*}
and $I^{h_2}_2$ is the integral involving the second vorticity $\varpi^p_2$. We remark that $I_a$ does not depend on $h_2$. The sign of $I^{h_2}_b$ is the same as the sign of $z_2$, thus we get $I_b^{h_2}<0$ and this is independent of the choice of $a$ and $h_2$. Now we fix $b$ and we take $h_2$ sufficiently large such that 
$$
I^{h_2}_b+I^{h_2}_2\leq I^{h_2}_b+\frac{2\pi}{\cosh(h_2/2)-1}\frac{1+\cosh(h_2)}{(\cosh(h_2)-1)^2}<0.
$$
We can do that because
$$
\frac{c_b\sinh(h_2/3)}{(\cosh(h_2/2)+1)^2} \leq |I^{h_2}_b|
$$
%\leq \frac{c_b\sinh(h_2/2)}{(\cosh(h_2/3)-1)^2}, 
or, equivalently,
$$
I^{h_2}_b+I^{h_2}_2=-|I^{h_2}_b|+I^{h_2}_2\leq -\frac{c_b\sinh(h_2/3)}{(\cosh(h_2/2)-1)^2}+\frac{2\pi}{\cosh(h_2/2)-1}\frac{1+\cosh(h_2)}{(\cosh(h_2)-1)^2}<0,
$$
if $h_2$ is large enough. The integral $I_a$ is well defined and positive, but goes to zero as $a$ grows. Then, fixed $b$ and $h_2$ in such a way $I^{h_2}_b+I^{h_2}_2<0$, we take $a$ sufficiently large such that $I_a+I^{h_2}_b+I^{h_2}_2<0$. We are done with the periodic case.

We proceed with the flat at infinity case. We take $2<b\leq a$ as before and $0<\delta<1$ and define
\begin{equation}\label{IIzRI}
z_1(\alpha)=\alpha-\sin(\alpha)\exp(-\alpha^2),
\end{equation}
and 
\begin{equation}\label{IIzR}
z_2(\alpha)=\left\{
\begin{array}{lllll}\displaystyle \frac{\sin(a\alpha)}{a} & \hbox{ if }  \displaystyle 0\leq \alpha\leq \frac{\pi}{a},\\
\displaystyle \frac{\sin\left(\pi\frac{\alpha-(\pi/a)}{(\pi/a)-(\pi/b)}\right)}{b} & \text { if } \displaystyle\frac {\pi}{a}<\alpha<\frac{\pi}{b},\\
\displaystyle \left(\frac{-h_2^\delta}{\frac{\pi}{2}-\frac{\pi}{b}}\right)\left(\alpha -\frac{\pi}{b}\right) & \text{ if }\displaystyle  \frac{\pi}{b}\leq\alpha<\frac{\pi} {2},\\
\displaystyle -\left(\frac{-h_2^\delta}{\frac{\pi}{2}-\frac{\pi}{b}}\right)\left(\alpha -\pi+\frac{\pi}{b}\right) & \text{ if } 
\displaystyle\frac{\pi}{2}\leq\alpha<\pi (1-\frac{1}{b}),\\
\displaystyle 0 & \text{ if }
\displaystyle \pi(1-\frac{1}{b})\leq\alpha.\\
                   \end{array}\right.
\end{equation}
We have
$$
h_2-h_2^\delta<h_2+z_2(\beta)<h_2+h_2^\delta,
$$
and we assume $1<h_2-h_2^\delta$. Inserting the curve \eqref{IIzRI} and \eqref{IIzR} in \eqref{IIw2f} and changing variables, we obtain 
\begin{eqnarray*}
|\varpi_2(\beta)|\leq 2\text{P.V.}\int_\RR\frac{(h_2+h_2^\delta)h_2^\delta\left(\frac{\pi}{2}-\frac{\pi}{b}\right)^{-1}}{(h_2-h_2^\delta)^2+(\gamma-\sin(\beta-\gamma)e^{-(\beta-\gamma)^2})^2}d\gamma.
\end{eqnarray*}
We split the integral in two parts:
$$J_1=2\text{P.V.}\int_{B(0,2(h_2-h_2^\delta))}\frac{(h_2+h_2^\delta)h_2^\delta\left(\frac{\pi}{2}-\frac{\pi}{b}\right)^{-1}}{(h_2-h_2^\delta)^2+(\gamma-\sin(\beta-\gamma)e^{-(\beta-\gamma)^2})^2}d\gamma\leq 8h_2^\delta\left(\frac{\pi}{2}-\frac{\pi}{b}\right)^{-1},
$$
and
$$
J_2=2\text{P.V.}\int_{B^c(0,2(h_2-h_2^\delta))}\frac{(h_2+h_2^\delta)h_2^\delta\left(\frac{\pi}{2}-\frac{\pi}{b}\right)^{-1}}{(h_2-h_2^\delta)^2+(\gamma-\sin(\beta-\gamma)e^{-(\beta-\gamma)^2})^2}d\gamma.
$$
We have
\begin{multline*}
K_1=\text{P.V.}\int_{2(h_2-h_2^\delta)}^\infty\frac{1}{(h_2-h_2^\delta)^2+(\gamma-\sin(\beta-\gamma)e^{-(\beta-\gamma)^2})^2}d\gamma\\
\leq\text{P.V.}\int_{2(h_2-h_2^\delta)}^\infty\frac{1}{(h_2-h_2^\delta)^2+\gamma^2-2\gamma\sin(\beta-\gamma)e^{-(\beta-\gamma)^2}}d\gamma\\
\leq \text{P.V.}\int_{2(h_2-h_2^\delta)}^\infty\frac{1}{(h_2-h_2^\delta-\gamma)^2+2\gamma(h_2-h_2^\delta-\sin(\beta-\gamma)e^{-(\beta-\gamma)^2})}d\gamma,
\end{multline*}
and using that $h_2$ is such that $1<h_2-h_2^\delta$, we get
$$
K_1\leq\text{P.V.}\int_{2(h_2-h_2^\delta)}^\infty\frac{1}{(h_2-h_2^\delta-\gamma)^2}d\gamma=\frac{1}{h_2-h_2^\delta}.
$$
The remaining integral is
\begin{multline*}
K_2=\text{P.V.}\int^{-2(h_2-h_2^\delta)}_{-\infty}\frac{1}{(h_2-h_2^\delta)^2+(\gamma-\sin(\beta-\gamma)e^{-(\beta-\gamma)^2})^2}d\gamma\\
\leq\text{P.V.}\int^{-2(h_2-h_2^\delta)}_{-\infty}\frac{1}{(h_2-h_2^\delta)^2+\gamma^2-2\gamma\sin(\beta-\gamma)e^{-(\beta-\gamma)^2}}d\gamma\\
\leq \text{P.V.}\int^{-2(h_2-h_2^\delta)}_{-\infty}\frac{1}{(h_2-h_2^\delta+\gamma)^2-2\gamma(h_2-h_2^\delta+\sin(\beta-\gamma)e^{-(\beta-\gamma)^2})}d\gamma,
\end{multline*}
and using that $h_2$ is such that $1<h_2-h_2^\delta$, we get
$$
K_2\leq\text{P.V.}\int^{-2(h_2-h_2^\delta)}_{-\infty}\frac{1}{(h_2-h_2^\delta+\gamma)^2}d\gamma=\frac{1}{h_2-h_2^\delta}.
$$
Putting all together we get
$$
J_2\leq 4h_2^\delta\left(\frac{\pi}{2}-\frac{\pi}{b}\right)^{-1},
$$
and
$$
|\varpi_2(\beta)|\leq 12h_2^\delta\left(\frac{\pi}{2}-\frac{\pi}{b}\right)^{-1}.
$$
Using this bound in \eqref{IIsing3R} we get
$$
|I_2^{h_2}|\leq 3h_2^{\delta-1}\left(\frac{\pi}{2}-\frac{\pi}{b}\right)^{-1}.
$$
Then, as before,
$$
\partial_\alpha v_1(0)\leq I_a+I_b^{h_2}+I_2^{h_2},
$$
where $I_a,I_b^{h_2}$ are the integrals $I_1(0)$ on the intervals $(0,\pi/a)$ and $((\pi/b+\pi)/3,(2\pi-\pi/b)/3)$, respectively. We have
$$
c_b\frac{2h_2^\delta}{3(h_2^\delta)^4}\leq|I_b^{h_2}|
$$
%\leq c_b\frac{h_2^\delta}{(2h_2^\delta/3)^4}d\beta,
thus,
$$
I^{h_2}_b+I^{h_2}_2=-|I^{h_2}_b|+I^{h_2}_2\leq -c_b\frac{2h_2^{-3\delta}}{3}+3h_2^{\delta-1}\left(\frac{\pi}{2}-\frac{\pi}{b}\right)^{-1}.
$$ 
To ensure that the decay of $I_2^{h_2}$ is faster than the decay of $I_b^{h_2}$ we take $\delta<1/4$. Now, fixing $b$, we can obtain $1<h_2$ and $0<\delta<1/4$ such that $1<h_2-h_2^\delta$ and $I^{h_2}_b+I^{h_2}_2<0$. Taking $a>>b$ we obtain a curve such that $\partial_\alpha v_1(0)<0$. In order to conclude the argument it is enough to approximate these curves \eqref{IIzT} and \eqref{IIzR} by analytic functions. We are done with this step of the proof.

\textbf{Showing the forward and backward solvability:} At this point, we need to prove that there is a solution forward and backward in time corresponding to these curves \eqref{IIzT} and \eqref{IIzR}. Indeed, if this solution exists then, due to the previous step, we obtain that, for a short time $t<0$, the solution is a graph with finite $H^3(\Omega)$ energy (in fact, it is analytic). This graph at time $t=0$ has a blow up for $\|\pax f\|_{L^\infty}$ and, for a short time $t>0$, the solution can not be parametrized as a graph. We show the result corresponding to the flat at infinity case, being the periodic one analogous. We consider curves $z$ satisfying the arc-chord condition and such that
$$
\lim_{|\alpha|\rightarrow\infty}|z(\alpha)-(\alpha,0)|=0.
$$
We define the complex strip $\BB_r=\{\zeta+i\xi, \zeta\in\RR, |\xi|<r\},$ and the spaces 
\begin{equation}\label{IIXdef}
X_r=\{z=(z_1,z_2) \text{ analytic curves satisfying the arc-chord condition on } \BB_r\},
\end{equation}
with norm
$$
\|z\|^2_r=\|z(\gamma)-(\gamma,0)\|^2_{H^3(\BB_r)},
$$
where $H^3(\BB_r)$ denotes the Hardy-Sobolev space on the strip with the norm 
\begin{equation}\label{IIXnorm}
\|f\|^2_r=\sum_{\pm}\int_\RR|f(\zeta\pm r i)|^2d\zeta+\int_\RR|\paa^3 f(\zeta\pm r i)|^2d\zeta,
\end{equation}
(see \cite{bakan2007hardy}). These spaces form a Banach scale. For notational convenience we write $\gamma=\alpha\pm ir$, $\gamma'=\alpha\pm ir'$. Recall that, for $0<r'<r$, 
\begin{equation}\label{IIcauchy2}
\|\partial_\alpha \cdot\|_{L^2(\BB_{r'})}\leq\frac{C}{r-r'}\|\cdot\|_{L^2(\BB_r)}.
\end{equation}
We consider the complex extension of \eqref{IIeq7} and \eqref{IIeq8}, which is given by
\begin{multline}
\label{IIeq18}
\pat z(\gamma)=\text{P.V.}\int_\RR\frac{(z_1(\gamma)-z_1(\gamma-\beta))(\paa z(\gamma)-\paa z(\gamma-\beta))}{(z_1(\gamma)-z_1(\gamma-\beta))^2+(z_2(\gamma)-z_2(\gamma-\beta))^2}d\beta\\
+\frac{1}{2\pi}\text{P.V.}\int_\RR\frac{\varpi_2(\gamma-\beta)(z(\gamma)-h(\gamma-\beta))^\perp}{(z_1(\gamma)-(\gamma-\beta))^2+(z_2(\gamma)+h_2)^2}d\beta\\
+\paa z(\gamma)\frac{1}{2\pi}\text{P.V.}\int_\RR\frac{(z_2(\gamma)+h_2)\varpi_2(\gamma-\beta)}{(z_1(\gamma)-(\gamma-\beta))^2+(z_2(\gamma)+h_2)^2}d\beta,
\end{multline}
with
\begin{equation}
\label{IIeq17}
\varpi_2(\gamma)=2\mathcal{K}\,\text{P.V.}\int_\RR\frac{(h_2+z_2(\gamma-\beta))\paa z_2(\gamma-\beta)}{(\gamma-z_1(\gamma-\beta))^2+(h_2+z_2(\gamma-\beta))^2}d\beta.
\end{equation}
Recall the fact that in the case of a real variable graph $\varpi_2$ has the same regularity as $f$, but in the case of an arbitrary curve $\varpi_2$ is, roughly speaking, at the level of the first derivative of the interface. This fact will be used below. We define
\begin{equation}
\label{IIeq19}
d^-[z](\gamma,\beta)=\frac{\beta^2}{(z_1(\gamma)-z_1(\gamma-\beta))^2+(z_2(\gamma)-z_2(\gamma-\beta))^2},
\end{equation}
\begin{equation}
\label{IIeq19b}
d^h[z](\gamma,\beta)=\frac{1+\beta^2}{(z_1(\gamma)-(\gamma-\beta))^2+(z_2(\gamma)+h_2)^2}.
\end{equation}
The function $d^-$ is the complex extension of the \emph{arc chord condition} and we need it to bound the terms with $\varpi_1$. The function $d^h$ comes from the different permeabilities and we use it to bound the terms with $\varpi_2$. We observe that both are bounded functions for the considered curves. Consider $0< r'<r$ and the set
$$
O_R=\{z\in X_r\text{ such that }\|z\|_r<R, \|d^-[z]\|_{L^\infty(\BB_r)}<R, \|d^h[z]\|_{L^\infty(\BB_r)}<R\},  
$$
where $d^-[z]$ and $d^h[z]$ are defined in \eqref{IIeq19} and \eqref{IIeq19b}. Then we claim that, for $z,w\in O_R$, the righthand side of \eqref{IIeq18}, $F:O_R\rightarrow X_{r'}$ is continuous and the following inequalities holds:
\begin{eqnarray}\label{IIeq20}&&\|F[z]\|_{H^3(\BB_{r'})}\leq\frac{C_R}{r-r'}\|z\|_r,\\
\label{IIeq21}&&\|F[z]-F[w]\|_{H^3(\BB_{r'})}\leq\frac{C_R}{r-r'}\|z-w\|_{H^3(\BB_r)},\\
\label{IIeq22}&&\sup_{\gamma\in \BB_r,\beta\in\RR}|F[z](\gamma)-F[z](\gamma-\beta)|\leq C_R|\beta|.
\end{eqnarray}
The claim for the spatial operator corresponding to $\varpi_1$ has been studied in \cite{ccfgl}, thus, we only deal with the new terms containing $\varpi_2$. For the sake of brevity we only bound some terms, being the other analogous. Using Tonelli's theorem and Cauchy-Schwartz inequality we have that
$$
\|\varpi_2\|_{L^2(\BB_{r'})}\leq c\|d^h[z]\|_{L^\infty}(1+\|z_2\|_{L^\infty(\BB_{r'})})\|\paa z_2\|_{L^2(\BB_{r'})}.
$$
Moreover, we get 
\begin{equation}\label{IIeqw2}
\|\varpi_2\|_{H^2(\BB_{r})}\leq C_R\|z\|_r.
\end{equation}
For $\paa^3 \varpi_2$ the procedure is similar but we lose one derivative. Using \eqref{IIcauchy2} and Sobolev embedding we conclude 
\begin{equation}\label{IIeqw2b}
\|\varpi_2\|_{H^3(\BB_{r'})}\leq \frac{C_R}{r-r'}\|z\|_r.
\end{equation}
From here inequality \eqref{IIeq20} follows. Inequality \eqref{IIeq21}, for the terms involving $\varpi_1$, can be obtained using the properties of the Hilbert transform as in \cite{ccfgl}. Let's change slightly the notation and write $\varpi_2[z](\gamma)$ for the integral in \eqref{IIeq17}. We split
\begin{eqnarray*}
A_1&=&\text{P.V.}\int_\RR\frac{(\varpi_2[z](\gamma'-\beta)-\varpi_2[w](\gamma'-\beta))(z(\gamma')-h(\gamma'-\beta))^\perp}{(z_1(\gamma')-(\gamma'-\beta))^2+(z_2(\gamma')+h_2)^2}d\beta\\
&&+\text{P.V.}\int_\RR\frac{\varpi_2[w](\gamma'-\beta)\left((z(\gamma')-h(\gamma'-\beta))^\perp-(w(\gamma')-h(\gamma'-\beta))^\perp\right)}{(z_1(\gamma')-(\gamma'-\beta))^2+(z_2(\gamma')+h_2)^2}d\beta\\
&&+\text{P.V.}\int_\RR\varpi_2[w](\gamma'-\beta)(w(\gamma')-h(\gamma'-\beta))^\perp\frac{d^h[z](\gamma',\beta)-d^h[w](\gamma',\beta)}{1+\beta^2}d\beta\\
&&=B_1+B_2+B_3.
\end{eqnarray*}
In $B_3$ we need some extra decay at infinity to ensure the finiteness of the integral. We compute
$$
|d^h[z]-d^h[w]|\leq C_R\frac{|d^h[z]d^h[w]|}{1+\beta^2}\left|(1+\beta)(z_1-w_1)+z_2-w_2\right|<C_R|z-w|\frac{|1+\beta|}{1+\beta^2},
$$
and, due to Sobolev embedding, we get
$$
\|B_3\|_{L^2(\BB_{r'})}\leq C_R \|\varpi_2[w]\|_{L^2(\BB_{r'})}\|z-w\|_{L^\infty(\BB_{r'})}\leq\frac{C_R}{r-r'}\|z-w\|_{H^3(\BB_r)}. 
$$
For the second term we obtain the same bound
$$
\|B_2\|_{L^2(\BB_{r'})}\leq C_R \|\varpi_2[w]\|_{L^2(\BB_{r'})}\|z-w\|_{L^\infty(\BB_{r'})}\leq\frac{C_R}{r-r'}\|z-w\|_{H^3(\BB_r)}. 
$$
We split $B_1$ componentwise. In the first coordinate we have
\begin{multline*}
\|C_1\|_{L^2(\BB_{r'})}=\left\|\text{P.V.}\int_\RR\frac{(\varpi_2[z](\gamma'-\beta)-\varpi_2[w](\gamma'-\beta))(-z_2(\gamma')-h_2)}{(z_1(\gamma')-(\gamma'-\beta))^2+(z_2(\gamma')+h_2)^2}d\beta\right\|_{L^2(\BB_{r'})}\\
\leq C_R\|\varpi_2[z]-\varpi_2[w]\|_{L^2(\BB_{r'})}.
\end{multline*}
In the second coordinate we need to ensure the integrability at infinity. We get
\begin{eqnarray*}
C_2&=&\text{P.V.}\int_\RR\frac{(\varpi_2[z](\gamma'-\beta)-\varpi_2[w](\gamma'-\beta))(z_1(\gamma')-\gamma')}{(z_1(\gamma')-(\gamma'-\beta))^2+(z_2(\gamma')+h_2)^2}d\beta\\
&&+\text{P.V.}\int_\RR(\varpi_2[z](\gamma'-\beta)-\varpi_2[w](\gamma'-\beta))\left(\frac{\beta d^h[z]}{1+\beta^2}-\frac{1}{\beta}\right)d\beta\\
&&+H\varpi_2[z](\gamma')-H\varpi_2[w](\gamma'),
\end{eqnarray*}
and, with this splitting and the properties of the Hilbert transform, we obtain
$$
\|C_2\|_{L^2(\BB_{r'})}\leq C_R\|\varpi_2[z]-\varpi_2[w]\|_{L^2(\BB_{r'})}.
$$
We get
$$
\varpi_2[z]-\varpi_2[w]=C_3+C_4+C_5,
$$
where
$$
C_3=2\mathcal{K}\;\text{P.V.}\int_\RR\frac{(z_2(\gamma-\beta)-w_2(\gamma-\beta))\paa z_2(\gamma-\beta)}{(\gamma-z_1(\gamma-\beta))^2+(h_2+z_2(\gamma-\beta))^2}d\beta,
$$
$$
C_4=2\mathcal{K}\;\text{P.V.}\int_\RR\frac{(h_2+w_2(\gamma-\beta))(\paa z_2(\gamma-\beta)-\paa w_2(\gamma-\beta))}{(\gamma-z_1(\gamma-\beta))^2+(h_2+z_2(\gamma-\beta))^2}d\beta,
$$
and
$$
C_5=2\mathcal{K}\;\text{P.V.}\int_\RR(h_2+w_2(\gamma-\beta))\paa w_2(\gamma-\beta)\frac{d^h[z](\gamma-\beta,-\beta)-d^h[w](\gamma-\beta,-\beta)}{1+\beta^2}d\beta. 
$$
From these expressions we obtain
$$
\|C_3\|_{L^2(\BB_{r'})}\leq C_R\|z-w\|_{L^\infty}\|\paa z_2\|_{L^2(\BB_{r'})},
$$
$$
\|C_4\|_{L^2(\BB_{r'})}\leq C_R\|\paa(z-w)\|_{L^2(\BB_{r'})},
$$
and
$$
\|C_5\|_{L^2(\BB_{r'})}\leq C_R\|z-w\|_{L^\infty}\|\paa z_2\|_{L^2(\BB_{r'})}.
$$
Collecting all these estimates, and due to Sobolev embedding and \eqref{IIcauchy2} we obtain
$$
\|B_1\|_{L^2(\BB_{r'})}\frac{C_R}{r-r'}\|z-w\|_{H^3(\BB_r)}. 
$$
We are done with \eqref{IIeq21}. Inequality \eqref{IIeq22} is equivalent to the bound $|\pat\paa z|<C_R$. Such a bound for the terms involving $\varpi_2$ can be obtained from \eqref{IIeq19b} and \eqref{IIeqw2}. For instance
\begin{multline*}
A_2=\text{P.V.}\int_\RR\frac{\paa\varpi_2(\gamma-\beta)(z(\gamma)-h(\gamma-\beta))^\perp}{(z_1(\gamma)-(\gamma-\beta))^2+(z_2(\gamma)+h_2)^2}d\beta=\\
-\text{P.V.}\int_\RR\varpi_2(\gamma-\beta)\partial_\beta\left(\frac{(z(\gamma)-h(\gamma-\beta))^\perp}{(z_1(\gamma)-(\gamma-\beta))^2+(z_2(\gamma)+h_2)^2}\right)d\beta\\
\leq C_R\|\varpi_2\|_{H^2(\BB_r)}\|d^h[z]\|_{L^\infty}.
\end{multline*} 
The remaining terms can be handled in a similar way. Now we can finish with the forward and backward solvability step. Take $z(0)$ the analytic extension of $z$ in \eqref{IIzR} (\eqref{IIzT} for the periodic case). We have $z(0)\in X_{r_0}$ for some $r_0>0$, it satisfies the arc-chord condition and does not reach the curve $h$, thus, there exists $R_0$ such that $z(0)\in O_{R_0}$. We take $r<r_0$ and $R>R_0$ in order to define $O_R$ and we consider the iterates
$$
z_{n+1}=z(0)+\int_0^tF[z_n]ds,\;z_0=z(0),
$$
and assume by induction that $z_k\in O_R$ for $k\leq n$. Then, following the proofs in \cite{ccfgl,CGO,nirenberg1972abstract,nishida1977note}, we obtain a time $T_{CK}>0$ of existence. It remains to show that 
$$
\|d^-[z_{n+1}]\|_{L^\infty(\BB_r)},\|d^h[z_{n+1}]\|_{L^\infty(\BB_r)}<R,
$$
for some times $T_A, T_B>0$ respectively. Then we choose $T=\min\{T_{CK},T_A,T_B\}$ and we finish the proof. As $d^-$ has been studied in \cite{ccfgl} we only deal with $d^h$. Due to \eqref{IIeq20} and the definition of $z(0)$, we have
$$
(d^{+}[z_{n+1}])^{-1}>\frac{1}{R_0}-C_R(t^2+t),
$$
and, if we take a sufficiently small $T_B$ we can ensure that for $t<T_B$ we have $d^h[z_{n+1}]<R$. We conclude the proof of the Theorem.
\end{proof}

We observe that in the periodic case the curve $z$ is of the same order as $h_2$, so, even if $h_2>>1$, this result is not some kind of \emph{linearization}. The same result is valid if $\mathcal{K}<<1$ for any $h_2$ (see Theorem \ref{IIteo4}). Moreover, we have numerical evidence showing that for every $|\mathcal{K}|<1$ and $h_2=\pi/2$ (and not $h_2>>1$) there are curves showing turning effect.

\begin{NE}
There are curves such that for every $|\mathcal{K}|<1$ and $h_2=\pi/2$ turn over.
\end{NE}
Let us consider first the periodic setting. Recall the fact that $h_2=\pi/2$ and let us define
\begin{equation}\label{IIzTNE}
z_1(\alpha)=\alpha-\sin(\alpha), \;\;z_2(\alpha)=\frac{\sin(3\alpha)}{3}-\sin(\alpha)\left(e^{-(\alpha+2)^2}+e^{-(\alpha-2)^2}\right)\text{ for $\alpha\in\TT$}.
\end{equation}
Inserting this curve in \eqref{IIsing3T} we obtain that for any possible $-1<\mathcal{K}<1$, 
$$
I_1(0)+|I_2(0)|<0.
$$
In particular
$$
\paa v_1^p(0)=I_1(0)+I_2(0)<I_1(0)+|I_2(0)|<0.
$$
Let us introduce the algorithm we use. We need to compute 
$$
\paa v^p_1(0)=\int_0^\pi\mathcal{I}_1+\int_0^\pi\mathcal{I}_2,
$$ 
where $\mathcal{I}_i$ means the $i-$integral in \eqref{IIsing3T}. Recall that $\mathcal{I}_i$ is two times differentiable, so, we can use the sharp error bound for the trapezoidal rule. We denote $dx$ the mesh size when we compute the first integral. We approximate the integral of $\mathcal{I}_1$ using the trapezoidal rule between $(0.1,\pi)$. We neglect the integral in the interval $(0,0.1)$, paying with an error denoted by $|E^1_{PV}|=O(10^{-3})$. The trapezoidal rule gives us an error $|E^1_I|\leq\frac{dx^2(\pi-0.1)}{12}\|\paa^2 \mathcal{I}_1\|_{L^\infty}$. As we know the curve $z$, we can bound $\paa^2 \mathcal{I}_1$. We obtain
$$
|E^1_I|\leq dx^2\frac{(\pi-0.1)}{6}10^{5}.
$$
We take $dx=10^{-7}$. Putting all together we obtain 
\begin{equation*}
|E^1|\leq|E^1_{PV}|+|E^1_I|+\leq 3O(10^{-3})= O(10^{-2}). 
\end{equation*}
Then, we can ensure that
\begin{equation}\label{III1err}
\paa z_2(0)\int_0^\pi\frac{\paa z_1(\beta)\sin(z_1(\beta))\sinh(z_2(\beta))}{(\cosh(z_2(\beta))-\cos(z_1(\beta)))^2}d\beta\leq -0.7+|E^1|<-0.6.
\end{equation}
We need to control analytically the error in the integral involving $\varpi^p_2$. This second integral has the error coming form the numerical integration, $E^2_I$ and a new error coming from the fact that $\varpi^p_2$ is known with some error. We denote this new error as $E^2_\varpi$. Let us write $\tilde{dx}$ the mesh size for the second integral. Then, using the smoothness of $\mathcal{I}_2$, we have
$$
|E^2_I|\leq\frac{\tilde{dx}^2}{16}\|\varpi^p_2\|_{C^2}\leq\frac{\tilde{dx}^2}{4}\cdot50.
$$
We take $\tilde{dx}=10^{-4}.$ It remains the error coming from $\varpi_2^p$. The second vorticity, $\varpi_2^p$, is given by the integral \eqref{IIw2p}. We compute the integral \eqref{IIw2p} using the same mesh size as for $I_2$, $\tilde{dx}$. Thus, the errors are
$$
|E_\varpi^2|\leq O(10^{-3}),
$$  
Putting all together we have
$$
|E^2|\leq |E^2_I|+|E^2_\varpi|\leq O(10^{-2}),
$$
and we conclude 
\begin{equation}\label{III2err}
\left|\frac{\paa z_2(0)}{4\pi}\int_0^\pi\frac{(\varpi^p_2(\beta)+\varpi^p_2(-\beta))(-1+\cosh(h_2)\cos(\beta))}{(\cosh(h_2)-\cos(\beta))^2}d\beta\right|\leq 0.1+|E^2|<0.2.
\end{equation}
Now, using \eqref{III1err} and \eqref{III2err}, we obtain $\paa v_1^p(0)<0$, and we are done with the periodic case. We proceed with the flat at infinity case. We have to deal with the unboundedness of the domain so we define
\begin{equation}\label{IIzRNE}
z_1(\alpha)=\alpha-\sin(\alpha)\exp(-\alpha^2/100), \;\;z_2(\alpha)=\frac{\sin(3\alpha)}{3}-\sin(\alpha)\left(e^{-(\alpha+2)^2}+e^{-(\alpha-2)^2}\right)\textbf{1}_{\{|\alpha|<\pi\}}.
\end{equation}

Inserting this curve in \eqref{IIsing3R} we obtain that for any possible $-1<\mathcal{K}<1$, 
$$
I_1(0)+|I_2(0)|<0.
$$
Then, as before,
$$
\paa v_1(0)=I_1(0)+I_2(0)<I_1(0)+|I_2(0)|<0.
$$
The function $z_2$ is Lipschitz, so the same for $\mathcal{I}_1$, where now $\mathcal{I}_i$ are the expressions in \eqref{IIsing3R} and the second integral $I_2$ is over an unbounded interval. To avoid these problems we compute the numerical aproximation of
$$
\int_{0.1}^{\pi-dx}\mathcal{I}_1+\int_0^{L_2}\mathcal{I}_2.
$$ 
Recall that $\varpi_2$ is given by \eqref{IIw2f} and then, due to the definition of $z_2$, we can approximate it by an integral over $(0,\pi-\tilde{dx})$. The lack of analyticity of $z_2$ and the truncation of $I_2(0)$ introduces two new sources of error. We denote them by $E^1_{z_2}$ and $E^2_\RR$. We take $dx=10^{-7},\tilde{dx}=10^{-4}$ and $L_2=2\pi$. Using the bounds $z_1\leq\pi$, $\paa z_1\leq 2$ and $z_2\leq h_2$ we obtain
$$
|E^1_{z_2}|\leq\left|\int_{\pi-dx}^\pi\mathcal{I}_1\right|\leq dx\cdot0.2\cdot4\pi^2\leq 8\cdot10^{-7}.
$$
We have
$$
|\varpi_2(\beta)|\leq 4\pi\frac{(h_2+\max_\gamma|z_2(\gamma)|)\max_\gamma|\paa z_2(\gamma)|}{\min_\gamma(h_2+z_2(\gamma))^2+(\beta-z_1(\gamma))^2}d\gamma\leq \frac{4\pi\cdot3\cdot2}{\min_\gamma(h_2+z_2(\gamma))^2+(\beta-z_1(\gamma))^2}=C(\beta),
$$
and we get $C(\beta)<C(L_2)$ for $\beta>L_2$. Using this inequality we get the desired bound for the second error as follows:
$$
|E^2_{\RR}|\leq \frac{|C(L_2)|}{\pi}\int_{L_2}^\infty\frac{\beta^2}{\left(\beta^2+\left(\frac{\pi}{2}\right)^2\right)^2}\leq \frac{4\pi\cdot3\cdot2}{10}\cdot 0.05<4\cdot 10^{-1}.
$$
The other errors can be bounded as before, obtaining,
$$
|E^1|\leq |E^1_{PV}|+|E^1_I|+|E^1_{z_2}|=O(10^{-2}),
$$
$$
|E^2|\leq |E^2_{\varpi}|+|E^2_I|+|E^2_{\RR}|=0.42.
$$
We conclude 
\begin{equation}\label{III1errf}
\paa z_2(0)\cdot4\text{P.V.}\int_0^\infty\frac{\paa z_1(\beta)z_1(\beta)z_2(\beta)}{((z_1(\beta))^2+(z_2(\beta))^2)^2}d\beta\leq -0.7+|E^1|<-0.6,
\end{equation}
and
\begin{equation}\label{III2errf}
\left|-\frac{1}{2\pi}\text{P.V.}\int_0^\infty\frac{(\varpi_2(\beta)+\varpi_2(-\beta))\beta^2}{(\beta^2+h_2^2)^2}d\beta\right|<0.02+|E^2|<0.5.
\end{equation}
Putting together \eqref{III1errf} and \eqref{III2errf} we conclude $\paa v_1(0)<0$. 

In order to complete a rigorous enclosure of the integral, we are left with the bounding of the errors coming from the floating point representation and the computer operations and their propagation. In a forthcoming paper (see \cite{GG}) we will deal with this matter. By using interval arithmetics we will give a computer assisted proof of this result.

\subsection{Finite depth}
In this section we show the existence of finite time singularities for some curves and physical parameters in an explicit range (see \eqref{IIeqK1}). This result is a consequence of Theorem 4 in \cite{CGO} by means of a continuous dependence on the physical parameters. As a consequence the range of physical parameters plays a role. Indeed, we have
\begin{teo}
Let us suppose that the Rayleigh-Taylor condition is satisfied, \emph{i.e.} $\rho^2-\rho^1>0$, and take $0<h_2<\frac{\pi}{2}$. There are $f_0(x)=f(x,0)\in H^3(\RR)$, an admissible (see Theorem \ref{IIteo3}) initial datum, such that, for any $|\mathcal{K}|$ small enough, there exists solutions of \eqref{IIeqv2} such that $\lim_{t\rightarrow T^*}\|\pax f(t)\|_{L^\infty}=\infty$ for $0<T^*<\infty$. For short time $t>T^*$ the solution can be continued but it is not a graph.
\label{IIteo4}
\end{teo}
\begin{proof}
The proof is similar to the proof in Theorem \ref{IIteo2}. First, using the result in \cite{CGO} we obtain a curve, $z(0)$, such that the integrals in $\paa v_1(0)$ coming from $\varpi_1$ have a negative contribution. The second step is to take $\mathcal{K}$ small enough, when compared with some quantities depending on the curve $z(0)$, such that the contribution of the terms involving $\varpi_2$ is small enough to ensure the singularity. Now, the third step is to prove, using a Cauchy-Kovalevsky theorem, that there exists local in time solutions corresponding to the initial datum $z(0)$. To simplify notation we take $\kappa^1(\rho^2-\rho^1)=4\pi$. Then the parameters present in the problem are $h_2$ and $\mathcal{K}$.

\textbf{Obtaining the correct expression:} As in \cite{CGO} and Theorem \ref{IIteo2} we obtain
$$
\paa v_1(0)=\pat\paa z_1(0)=I_1+I_2,
$$
where
$$
I_1=2\paa z_2(0)\int_0^\infty \frac{\paa z_1(\beta)\sinh(z_1(\beta))\sin(z_2(\beta))}{\left(\cosh(z_1(\beta))-\cos(z_2(\beta))\right)^2}+\frac{\paa z_1(\beta)\sinh(z_1(\beta))\sin(z_2(\beta))}{\left(\cosh(z_1(\beta))+\cos(z_2(\beta))\right)^2}d\beta,
$$
and
\begin{multline*}
I_2=\frac{\paa z_2(0)}{4\pi}\int_\RR\frac{\varpi_2(-\beta)(-\cosh(\beta)\cos(h_2)+1)}{(\cosh(\beta)-\cos(h_2))^2}d\beta\\
+\frac{\paa z_2(0)}{4\pi}\int_\RR\frac{\varpi_2(-\beta)(-\cosh(\beta)\cos(h_2)-\cos^2(h_2)+\sin^2(h_2))}{(\cosh(\beta)+\cos(h_2))^2}d\beta.
\end{multline*}

\textbf{Taking the appropriate curve and $\mathcal{K}$:}
From Theorem 4 in \cite{CGO} we know that there are initial curves $w_0$ such that $I_1=-a^2,$ $a=a(w_0)>0$. We take one of this curves and we denote this smooth, fixed curve as $z(0)$. We need to obtain $\mathcal{K}=\mathcal{K}(z(0),h_2)$ such that $\paa v_1(0)=-a^2+I_2<0$. As in \eqref{IIeq19b} we define
\begin{equation}\label{IIeq19d}
d_1^h[z](\gamma,\beta)=\frac{\cosh^2(\beta/2)}{\cosh(z_1(\gamma)-(\gamma-\beta))-\cos(z_2(\gamma)+h_2)},
\end{equation}
and
\begin{equation}\label{IIeq19f}
d_2^h[z](\gamma,\beta)=\frac{\cosh^2(\beta/2)}{\cosh(z_1(\gamma)-(\gamma-\beta))+\cos(z_2(\gamma)-h_2)}.
\end{equation}
From the definition of $I_2$ it is easy to obtain
$$
|I_2|\leq C(h_2)\paa z_2(0)\|\varpi_2\|_{L^\infty},
$$
where
$$
C(h_2)=\frac{1}{4\pi}\int_\RR\frac{\cosh(\beta)\cos(h_2)+1}{(\cosh(\beta)-\cos(h_2))^2}d\beta+\frac{1}{4\pi}\int_\RR\frac{\cosh(\beta)\cos(h_2)+\cos(2h_2)}{(\cosh(\beta)+\cos(h_2))^2}d\beta.
$$
From the definition of $\varpi_2$ for curves (which follows from \eqref{IIw2defc} in a straightforward way) we obtain
$$
\|\varpi_2\|_{L^\infty}\leq8\mathcal{K}\|\paa z_2\|_{L^\infty}\left(\|d_1^h[z]\|_{L^\infty}+\|d_2^h[z]\|_{L^\infty}\right)\left(1+\frac{\mathcal{K}}{\sqrt{2\pi}}\|G_{h_2,\mathcal{K}}\|_{L^1}\right).
$$
Fixing $0<h_2<\pi/2$ and collecting all the estimates we obtain
\begin{equation*}
|I_2|\leq C(h_2)8\paa z_2(0)\mathcal{K}\|\paa z_2\|_{L^\infty}\left(\|d_1^h[z]\|_{L^\infty}+\|d_2^h[z]\|_{L^\infty}\right)\left(1+\frac{\mathcal{K}}{\sqrt{2\pi}}\sup_{|\mathcal{K}|<1}\|G_{h_2,\mathcal{K}}\|_{L^1}\right).
\end{equation*}
Now it is enough to take 
\begin{equation}\label{IIeqK1}
|\mathcal{K}_1(z(0),h_2)|<\frac{(C(h_2)8\paa z_2(0)\|\paa z_2\|_{L^\infty})^{-1}a^2}{\left(\|d_1^h[z]\|_{L^\infty}+\|d_2^h[z]\|_{L^\infty}\right)\left(1+\frac{1}{\sqrt{2\pi}}\sup_{|\mathcal{K}|<1}\|G_{h_2,\mathcal{K}}\|_{L^1}\right)},
\end{equation}
to ensure that $\paa v_1(0)<0$ for this curve $z(0)$ and any $|\mathcal{K}|<|\mathcal{K}_1(z(0),h_2)|$.

\textbf{Showing the forward and backward solvability:}
We define
\begin{equation}\label{IIeq19c}
d^-[z](\gamma,\beta)=\frac{\sinh^2(\beta/2)}{\cosh(z_1(\gamma)-z_1(\gamma-\beta))-\cos(z_2(\gamma)-z_2(\gamma-\beta))},
\end{equation}
and
\begin{equation}\label{IIeq19e}
d^+[z](\gamma,\beta)=\frac{\cosh^2(\beta/2)}{\cosh(z_1(\gamma)-z_1(\gamma-\beta))+\cos(z_2(\gamma)-z_2(\gamma-\beta))}.
\end{equation}
Using the equations \eqref{IIeq19d},\eqref{IIeq19f},\eqref{IIeq19c} and \eqref{IIeq19e}, the proof of this step mimics the proof in Theorem \ref{IIteo2} and the proof in \cite{CGO} and so we only sketch it. As before, we consider curves $z$ satisfying the arc-chord condition and such that
$$
\lim_{|\alpha|\rightarrow\infty}|z(\alpha)-(\alpha,0)|=0.
$$
We define the complex strip $\BB_r=\{\zeta+i\xi, \zeta\in\RR, |\xi|<r\},$ and the spaces \eqref{IIXdef} with norm \eqref{IIXnorm} (see \cite{bakan2007hardy}). We define the set
\begin{multline*}
O_R=\{z\in X_r\text{ such that }\|z\|_r<R, \|d^-[z]\|_{L^\infty(\BB_r)}<R, \|d^+[z]\|_{L^\infty(\BB_r)}<R,\\
 \|d_1^h[z]\|_{L^\infty(\BB_r)}<R,\|d_2^h[z]\|_{L^\infty(\BB_r)}<R\},  
\end{multline*}
where $d_i^h[z]$ and $d^\pm[z]$ are defined in \eqref{IIeq19d}, \eqref{IIeq19f}, \eqref{IIeq19c} and \eqref{IIeq19e}, respectively. As before, we have that, for $z,w\in O_R$, complex extension of \eqref{IIeqv}, $F:O_R\rightarrow X_{r'}$ is continuous and the following inequalities holds:
\begin{eqnarray*}&&\|F[z]\|_{H^3(\BB_{r'})}\leq\frac{C_R}{r-r'}\|z\|_r,\\
&&\|F[z]-F[w]\|_{H^3(\BB_{r'})}\leq\frac{C_R}{r-r'}\|z-w\|_{H^3(\BB_r)},\\
&&\sup_{\gamma\in \BB_r,\beta\in\RR}|F[z](\gamma)-F[z](\gamma-\beta)|\leq C_R|\beta|.
\end{eqnarray*} We consider 
$$
z_{n+1}=z(0)+\int_0^tF[z_n]ds, \; z_0=z(0).
$$
Using the previous properties of $F$ we obtain that, for $T=T(z(0),R)$ small enough, $z^{n+1}\in  O_R,$ for all $n$. The rest of the proof follows in the same way as in \cite{nirenberg1972abstract, nishida1977note}.
\end{proof}

\bibliographystyle{abbrv}
%\bibliography{bibliografia}

\end{document}